\documentclass[amsmath,secnumarabic,floatfix,amssymb,nofootinbib,nobibnotes,letterpaper,12pt,tightenlines]{article}

\usepackage{geometry}
\usepackage{amssymb}
\usepackage{latexsym}
\usepackage{amsmath}
\usepackage{amscd}
\usepackage{amsthm}
\usepackage{txfonts}
\usepackage{graphicx}
\usepackage[dvipsnames]{xcolor}
\usepackage[percent]{overpic}
\usepackage{booktabs}
\usepackage{diagbox}
\usepackage{nicefrac}
\usepackage{supertabular}
\usepackage{cite}

\PassOptionsToPackage{caption=false,labelformat=empty}{subfig}
\usepackage[lofdepth]{subfig}
\usepackage[export]{adjustbox}
\usepackage{algorithm}
\usepackage{algorithmicx}
\usepackage{algpseudocode}

\usepackage{siunitx}
\usepackage{enumitem}
\usepackage{authblk}

\usepackage{fancybox}

\setlist{nosep}

\usepackage{multicol,ifthen,multirow}

\makeatletter
\let\mcnewpage=\newpage
\newcommand{\TrickSupertabularIntoMulticols}{%
\renewcommand\newpage{%
    \if@firstcolumn%
        \hrule width\linewidth height0pt%
            \columnbreak%
        \else%
          \mcnewpage%
        \fi%
}%
}
\makeatother

\def\figdir{Pictures/}
\graphicspath{\figdir}

\usepackage{aliascnt} %alias counter
\newcommand{\newautotheorem}[3] %{environment name}{counter}{displayed name} - Sets the autorefname of environment so that \autoref from the \hyperref-Packet works correctly.
{
\newaliascnt{#1}{#2}
\newtheorem{#1}[#1]{#3}
\aliascntresetthe{#1}
\expandafter\def\csname #1autorefname\endcsname{%
#3%
}%
}

\newtheorem{theorem}{Theorem}

\newautotheorem{lemma}{theorem}{Lemma}
\newautotheorem{proposition}{theorem}{Proposition}
\newautotheorem{corollary}{theorem}{Corollary}

\theoremstyle{definition}
\newautotheorem{definition}{theorem}{Definition}

\newautotheorem{conjecture}{theorem}{Conjecture}
\newautotheorem{remark}{theorem}{Remark}
\newautotheorem{claim}{theorem}{Claim}

\newtheorem*{sticknumber}{Theorem~\ref*{thm:exact-sticks} (restated)}
\newtheorem*{asymptotics}{Theorem~\ref*{thm:asymptotic picture} (restated)}

\newcommand{\R}{\mathbb{R}}

\newcommand{\stick}{\operatorname{stick}}

\newcommand{\superbridge}{\operatorname{sb}}
\newcommand{\bridge}{\operatorname{b}}

\newcommand{\snappy}{\texttt{SnapPy}}
\newcommand{\pyknotid}{\texttt{pyknotid}}
\newcommand{\knoodle}{\texttt{Knoodle}}

\newcommand{\ceq}{\coloneqq}

\newcommand{\newbound}[1]{$\langle$\! {#1}\! $\rangle$}
\newcommand{\known}[1]{\fbox{#1}}

\let\mgp=\marginpar \marginparwidth18mm \marginparsep1mm
\def\marginpar#1{\mgp{\raggedright\tiny #1}}

\let\lbl=\label
\def\label#1{\lbl{#1}\ifinner\else\marginpar{\ref{#1} #1}\ignorespaces\fi}

\usepackage[
	backref  = false%
	,pagebackref = false%
	,bookmarks  = true%
	,bookmarksdepth=3%
]{hyperref}
\usepackage[all]{hypcap}  
\usepackage[nameinlink]{cleveref}

\title{New upper bounds for stick numbers}
\author[$\ast$]{Jason Cantarella}
\author[$\dag$]{Andrew Rechnitzer}
\author[$\ddag$]{Henrik Schumacher}
\author[$\mathsection$]{Clayton Shonkwiler}
\affil[$\ast$]{Mathematics Department, University of Georgia, Athens, GA, USA}
\affil[$\dag$]{Mathematics Department, University of British Columbia, Vancouver, BC, Canada}
\affil[$\ddag$]{Institut für Mathematik, RWTH Aachen University, Aachen, Germany}
\affil[$\mathsection$]{Department of Mathematics, Colorado State University, Fort Collins, CO, USA}
\date{}

\begin{document}

\maketitle

% \keywords{closed random walk, random polygon, random knot, polymer models}

\begin{abstract}
	We use a version of simulated annealing with knot-type preserving moves to find polygonal representatives of various knot types with low stick number. These give better bounds on stick numbers of prime knots through 10 crossings, and for the first time give a comprehensive table of stick number bounds on all knots through 13 crossings. These are equal to existing lower bounds (and hence determine the stick number exactly) for 19 knot types whose exact stick number was not known previously.
\end{abstract}

\section{Introduction}

The \emph{stick number} $\stick[K]$ of a knot type $K$ is the minimum number of edges in any polygon in $\R^3$ with knot type $K$. Stick numbers have been studied since the late 1980s~\cite{randellMolecularConformationSpace1988,randellConformationSpacesMolecular1988,negamiRamseyTheoremsKnots1991,calvoGeometricKnotTheory1998,calvoCharacterizingPolygonsBbb2002} as a measure of knot complexity with a direct connection to the three-dimensional geometry of curves of the knot type. Like crossing number and bridge number, stick number can be computed for a substantial number of knots. Like ropelength~\cite{grosbergFlorytypeTheoryKnotted1996,buckThicknessCrossingNumber1999,diaoLowerBoundsLengths2003,ashtonKnotTighteningConstrained2011,millettPhysicalKnotTheory2012}, it has a natural connection to the behavior of real entangled systems. This combination makes stick number an appealing area for study.

This paper reports on the results of a numerical experiment. We start with polygons with known knot types and use topology-preserving moves to minimize edge number by simulated annealing. We find low stick number configurations for all knots through 13 crossings and for some twist knots, torus knots, pretzel knots, and randomly selected alternating and non-alternating knots with up to 64 crossings. This gives upper bounds for the stick numbers of the corresponding knot types.
We suspect that these bounds are fairly tight; our code finds configurations that achieve the stick number for \emph{all} knots with fewer than 100 crossings for which the stick numbers are known exactly (mostly small knots and some torus knots).

We compute the stick numbers of 19 knot types by observing that our configurations have edge count equal to an existing lower bound:
\begin{theorem}\label{thm:exact-sticks}
	The following knots have stick number equal to 10:
	\begin{itemize}
		\item $11n_i$ with $i \in \{71,75,76,78\}$;
		\item $13n_j$ with $j \in \{225, 230, 285, 288, 307, 584, 586, 593\}$\\
		      and $j \in \{602, 603, 604, 607, 608, 1192, 5018\}$.
	\end{itemize}
\end{theorem}

Our 10-stick $11n_{78}$ and $13n_{5018}$ are shown in \autoref{fig:10 stick exact}.

\begin{figure}[t]
	\centering
	\subfloat[$11n_{78}$]{\includegraphics[height=2in,width=2in,keepaspectratio,valign=b]{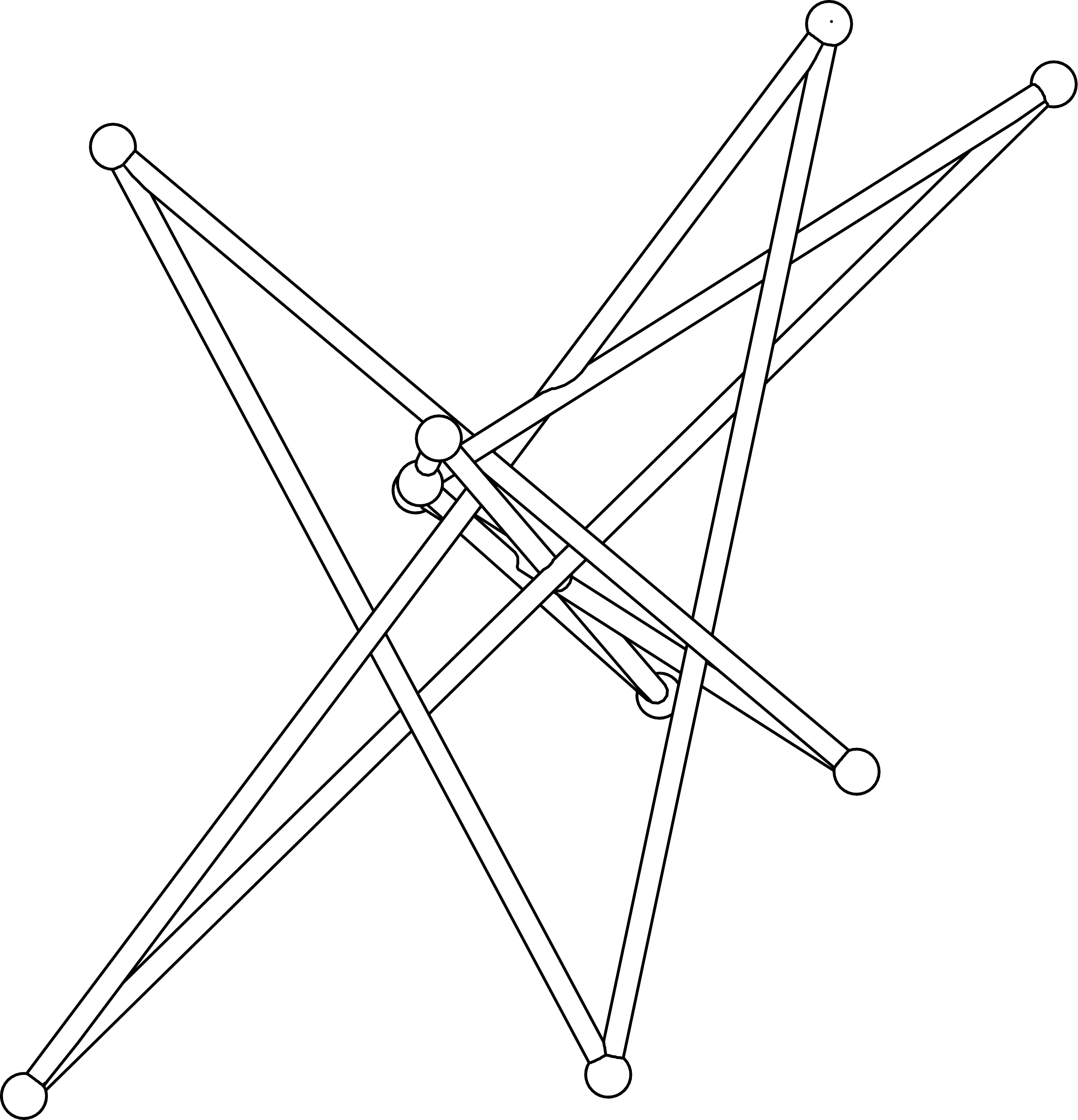}} \qquad \qquad
	\subfloat[$13n_{5018}$]{\includegraphics[height=2in,width=2in,keepaspectratio,valign=b]{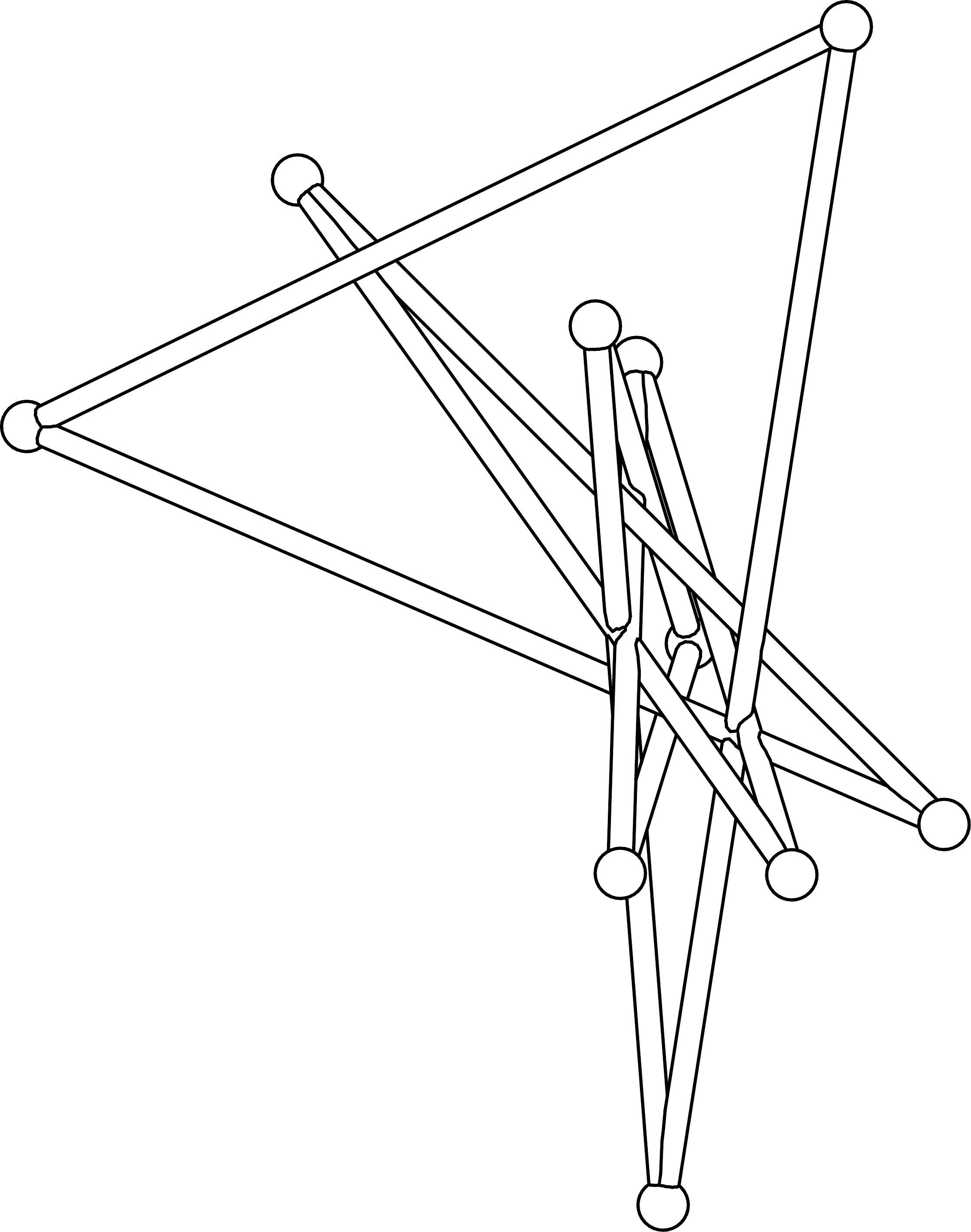}}
	\caption{10-stick realizations of $11n_{78}$ (left) and $13n_{5018}$ (right). Both are shown in orthographic perspective, viewed from the direction of the positive $z$-axis relative to the vertex coordinates given in the dataset~\cite{dataverse-table}.}
	\label{fig:10 stick exact}
\end{figure}

We improve the best known upper bounds for stick numbers of $3$ of the $9$-crossing knots (including $9_{36}$, shown in \autoref{fig:9_36 and 10_37}), $18$ of the $10$-crossing knots (including $10_{37}$, shown in \autoref{fig:9_36 and 10_37}), $62.5\%$ of the $11$-crossing knots, $83.2\%$ of the $12$-crossing knots, and $99.9\%$ of the $13$-crossing knots.
A table of stick number bounds for knots with crossing number $\leq 10$ is given in~\Cref{sec:rolfsen table} and the extended table for all knots through 13 crossings is included in supplementary data and has been added to KnotInfo~\cite{knotinfo}. Vertex coordinates of all of our polygons are available in a dataset distributed through the Harvard Dataverse~\cite{dataverse-table}.

\begin{figure}[t]
	\centering
	\subfloat[$9_{36}$]{\includegraphics[height=2in,width=2in,keepaspectratio,valign=b]{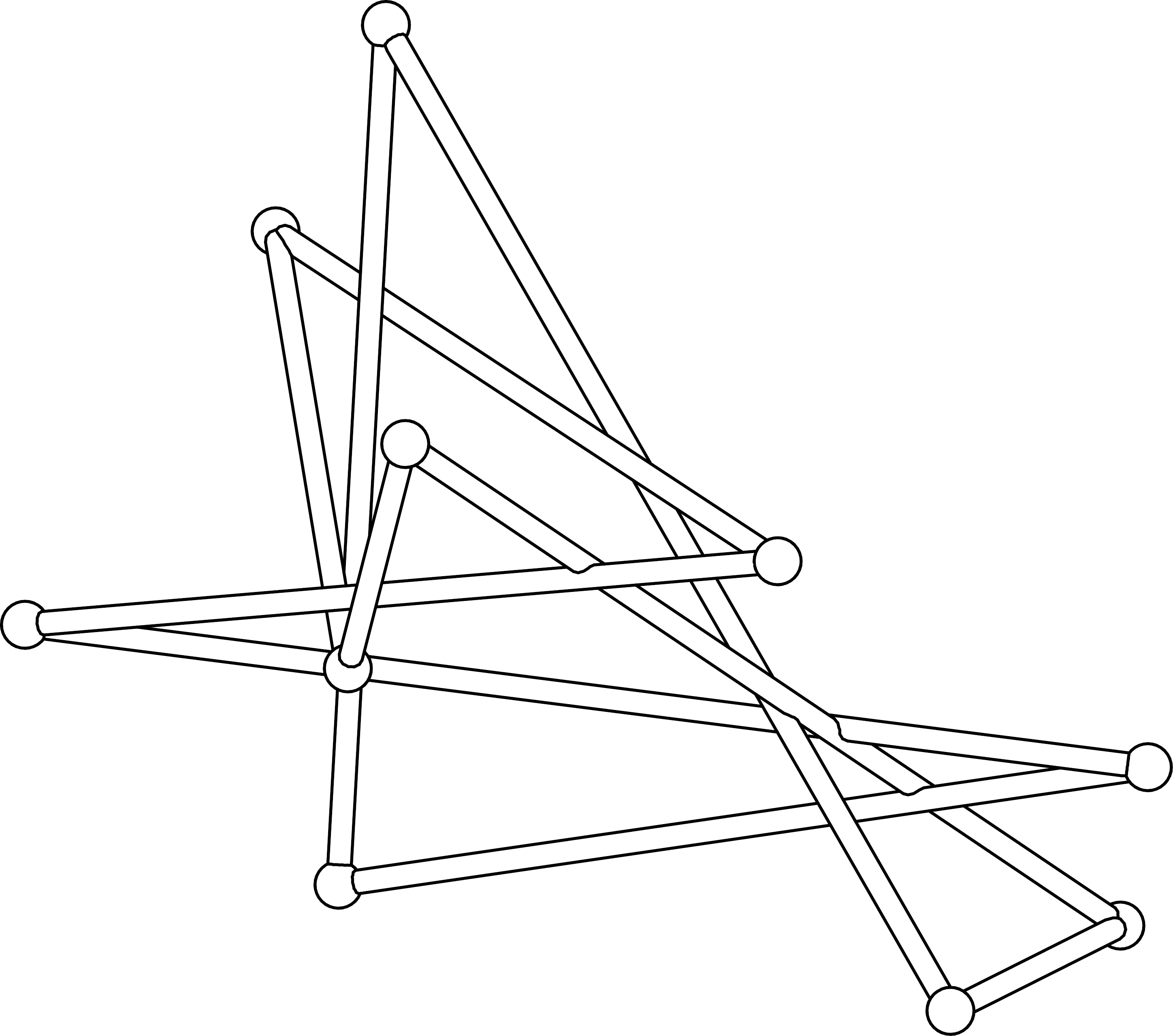}} \qquad \qquad
	\subfloat[$10_{37}$]{\includegraphics[height=2in,width=2in,keepaspectratio,valign=b]{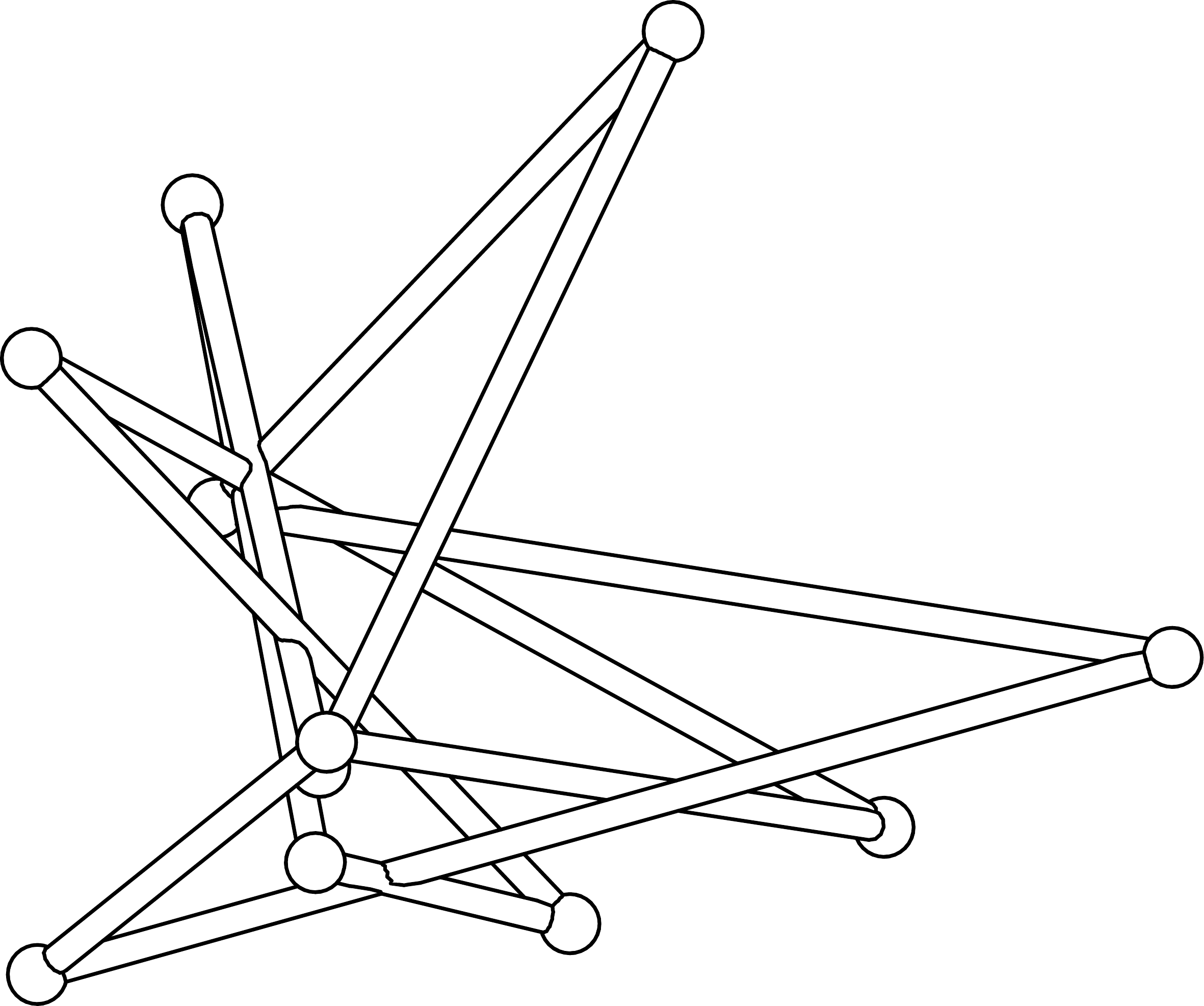}}
	\caption{A 10-stick $9_{36}$ (left) and an 11-stick $10_{37}$. For both knots this improves the previous best upper bound on stick number by 1. Both are shown in orthographic perspective, viewed from the direction of the positive $z$-axis relative to the vertex coordinates given in the dataset~\cite{dataverse-table}.}
	\label{fig:9_36 and 10_37}
\end{figure}

It is interesting to consider the asymptotic relationship between stick number and crossing number. In~\Cref{sec:results}, we combine several previous results to prove:
\begin{theorem}\label{thm:asymptotic picture}
	There are constants $\alpha_*$ and $\beta_*$ so that for any $\alpha < \alpha_*$ and $\beta > \beta_*$ there is some $C$ so that for all knot types with crossing number $c[K] > C$, we have
	\[
		\alpha \sqrt{c[K]} \leq \stick[K] \leq \beta \, c[K].
	\]
	Further, $\alpha_* \in [\!\sqrt{2},2]$ while $\beta_* \in [\nicefrac{2}{3}, \nicefrac{3}{2}]$.
\end{theorem}
We were able to show $\stick[K] \leq c[K]$ for all the knots we examined with at least twelve crossings. This leads us to conjecture that $\beta_* \leq 1$; see~\Cref{sec:discussion} and~\autoref{conj:stick bound}. It might be the case that $\beta_* = 1$ is the lowest value for which the theorem is true. We ran our algorithm on $(2,3, \dotsc, 3)$ pretzel knots with up to $98$ crossings and were unable to find configurations with $\stick[K] < c[K]-1$.\footnote{We note that in the course of proving~\autoref{thm:asymptotic picture}, we established that these pretzel knots have $\stick[K] \geq \frac{2}{3} c[K]$.}

\section{Methods}
\label{sec:methods}

There are two basic approaches to proving that $\stick[K] \leq n$; one can search the space of $n$-gons for knots of type $K$~(cf. \cite{millettKnottingRegularPolygons1994,calvoMinimalEdgePiecewise1998,millettMonteCarloExplorations2000,eddyImprovedStickNumber2019,eddyNewStickNumber2022,blairKnotsExactly102020,shonkwilerNewComputationsSuperbridge2020,shonkwilerAllPrimeKnots2022,shonkwilerNewSuperbridgeIndex2022a}) or one can search the space of polygons of fixed knot type $K$ for $n$-gons~(cf. \cite{schareinInteractiveTopologicalDrawing1998,rawdonUpperBoundsEquilateral2002}). In this experiment, we take the second approach; although this makes the search space much larger (in principle, infinite), we transform the search problem into an optimization problem on the number of edges. We find that this reduces the computation time by several orders of magnitude, while providing results that are always equal to or better than those found by the first approach.

% we may start with a planar diagram (PD) code for a diagram with knot type $K$ which specifies the diagram combinatorially as an embedded planar graph with over and under-crossings. We then use~\texttt{Knoodle}~\cite{knoodle} to find a corresponding lattice embedding with the same topology using the~\texttt{REAPR} embedding~\cite{REAPR}.

\begin{figure}
	\centering
	\subfloat[Lattice $9_{29}$]{\includegraphics[height=2in,width=2in,keepaspectratio,valign=b]{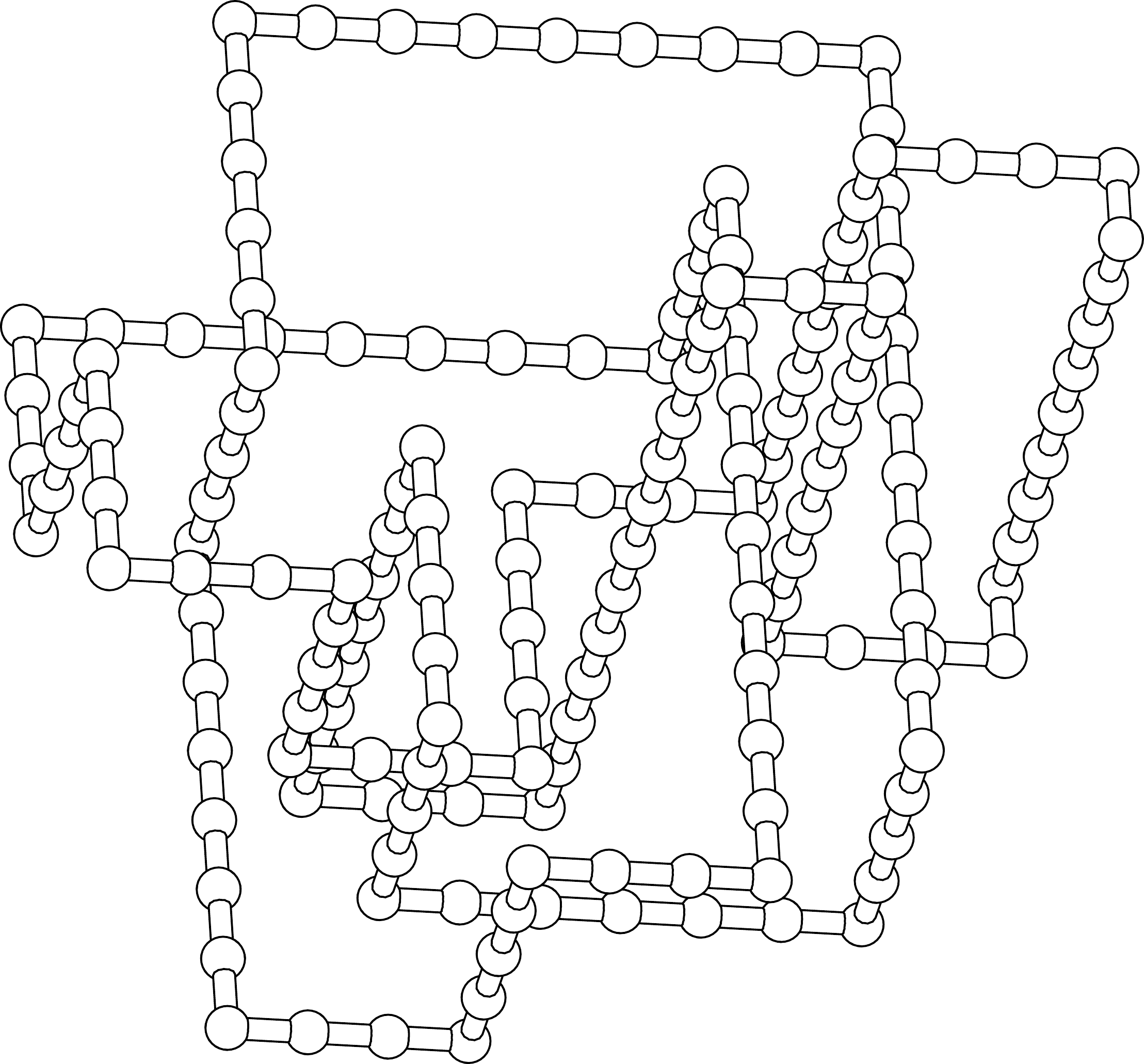}} \hfil
	\subfloat[Near-minimal lattice $9_{29}$]{\includegraphics[height=2in,width=2in,keepaspectratio,valign=b]{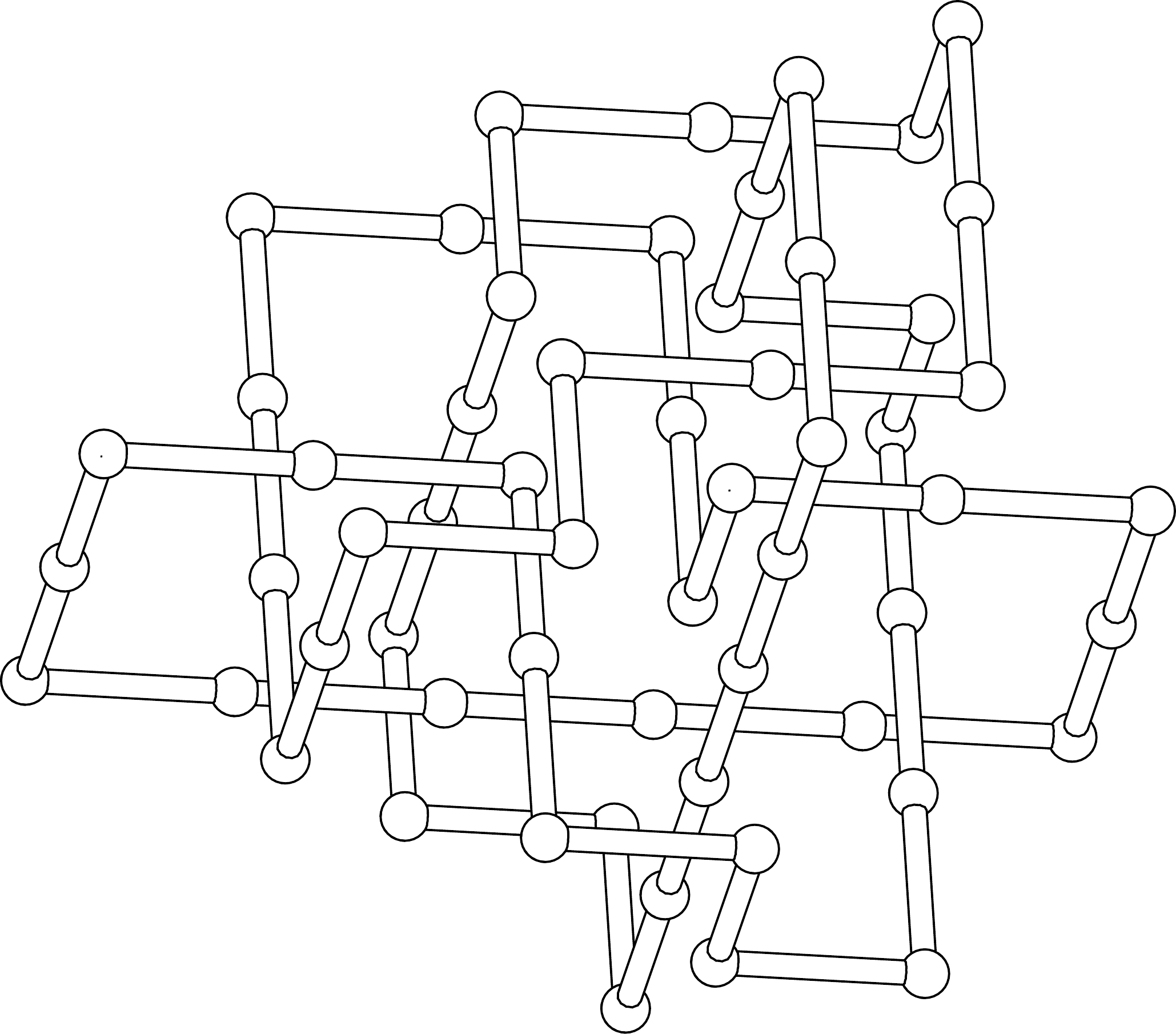}}

	\subfloat[Equilateral $9_{29}$]{\includegraphics[height=1.8in,width=1.8in,keepaspectratio,valign=b]{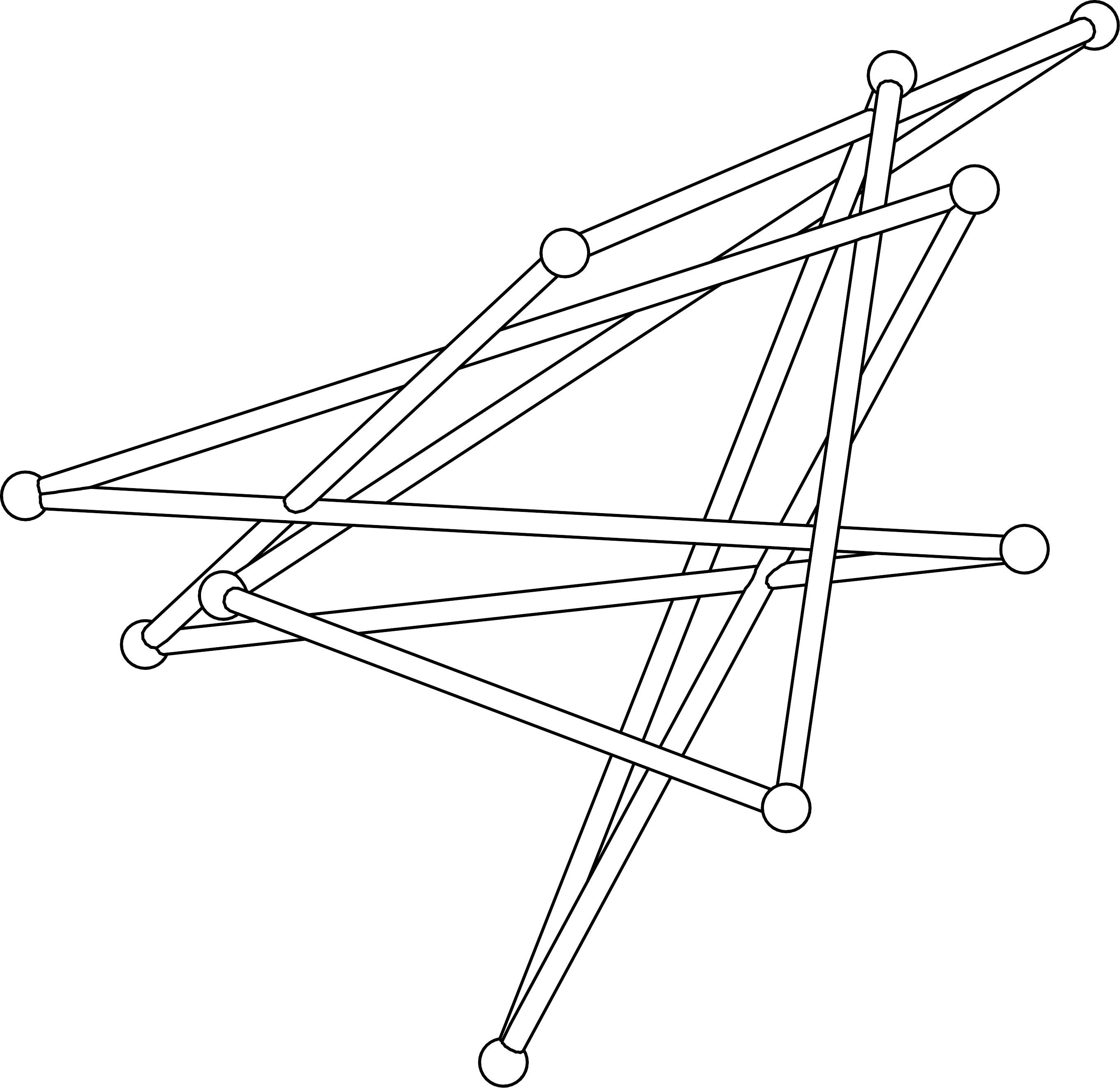}} \hfil
	\subfloat[Minimal $9_{29}$]{\includegraphics[height=1.8in,width=1.8in,keepaspectratio,valign=b]{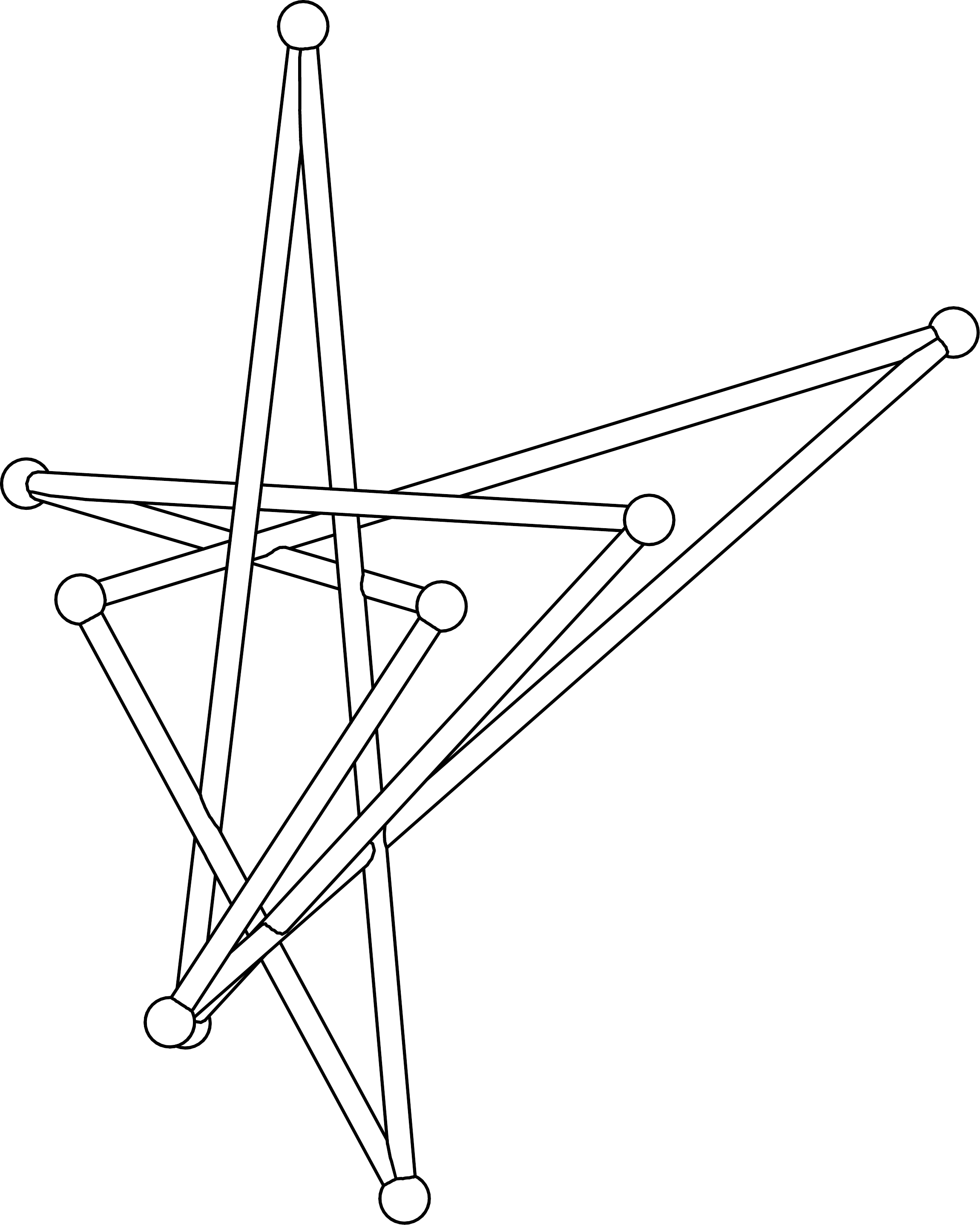}}
	\caption{An illustration of our algorithm, as applied to the knot $9_{29}$. Starting from a non-minimal lattice configuration with 168 edges (top left), simulated annealing with BFACF moves produces a nearly-minimal lattice configuration with 58 edges (top right).
	The off-lattice simulated annealer produces a 10-stick equilateral configuration (bottom left), and then the non-equilateral moves eventually reduce the number of sticks to $9$ (bottom right). This happens to be the minimum ~\cite{schareinInteractiveTopologicalDrawing1998,calvoGeometricKnotTheory1998}.
	In general, none of these configurations are guaranteed to be minimal: indeed, a 56-edge lattice $9_{29}$ is known~\cite{andrew-minimal-knots} and, as we will see in \autoref{thm:eq 9_29}, it is possible to find an equilateral embedding of $9_{29}$ with just $9$ sticks.}
	\label{fig:overview}
\end{figure}

To find a low-stick number representative of a knot type, we start with a lattice embedding of the knot. We then reduce the number of lattice edges by a version of simulated annealing using moves---known as BFACF moves~\cite{bergRandomPathsRandom1981, aragaodecarvalhoNewMonteCarloApproach1983, aragaodecarvalhoPolymers$g|varphi|^4$Theory1983}---which both preserve knot type~\cite{jansevanrensburgBFACFAlgorithmKnotted1991} and are fast to compute.
We will explain these moves in more detail below.
Once we have arrived at a nearly minimal lattice configuration, 
we switch to an off-lattice simulated annealer which uses topology-preserving fold moves and a new move which increases or decreases the number of edges (described below).
The final result is the configuration with fewest edges found by this process. 
We conclude by double-checking the knot type of the final polygon. 
The whole process is illustrated for a specific example in~\autoref{fig:overview}.

%\url{https://personal.math.ubc.ca/~andrewr/knots/minimal_knots.html}

%  and in service of work on computing knot probabilities~\cite{jansevanrensburgGeneralizedAtmosphericSampling2009,jansevanrensburgBFACFstyleAlgorithmsPolygons2011,jansevanrensburgUniversalityKnotProbability2011}

We now describe our algorithm and the experiments we performed in more detail (readers satisfied with the discussion above may safely skip to the next section). 

Our algorithm starts with a lattice embedding of the desired knot but does not specify how that embedding is to be found. For the experiments described in this paper, we used the following methods to construct initial embeddings. For knots through 10 crossings, the second author had previously found lattice minimal configurations~\cite{andrew-minimal-knots} (building on~\cite{schareinBoundsMinimumStep2009}), so we started there. For knots with 11–13 crossings, we started with the off-lattice polygonal realizations from KnotInfo~\cite{knotinfo} and approximated them by lattice polygons using the ``\texttt{embed clk}'' command in~\texttt{KnotPlot}~\cite{KnotPlot}.
For all other knots, we started with a planar diagram code. We then used the ``\texttt{REAPR}'' embedding function from \knoodle~\cite{knoodle}. In this method, which we describe in more detail elsewhere~\cite{REAPR}, the $x$- and $y$-coordinates are determined by an orthogonal graph layout of the knot diagram and the $z$-coordinates are determined by an integer-valued height function that minimizes total variation subject to the constraints imposed by the over/under relations at the crossings.

%Once we had (non-minimal) cubic lattice configurations we then used a variant of the BFACF algorithm~\cite{bergRandomPathsRandom1981, aragaodecarvalhoNewMonteCarloApproach1983, aragaodecarvalhoPolymers$g|varphi|^4$Theory1983} algorithm to produce a near-minimal lattice embedding of each knot.
%The BFACF algorithm is relatively easy to code and is known~\cite{jansevanrensburgBFACFAlgorithmKnotted1991} to conserve the topology of the embedded curve.
%This avoids complicated and expensive topology-preservation checks while making large reductions in length. Indeed, each elementary ``move'' of the algorithm runs in constant time. By contrast, the off-lattice 
%algorithm described in the next section requires $O(n)$ time per step, so it is advantageous
%to reduce as much as possible on the lattice.

To minimize the length of our lattice embeddings we use a version of simulated annealing~(cf. \cite{ubertiMinimalLinksCubic1998}).
At each iteration, we choose an edge in the embedding and a lattice face adjacent to that edge. 
Then we see if replacing the edges incident to the face with their complement (see Figure~\ref{fig: bfacf1}) causes the loop to self-intersect. 
(We can do this in constant time by storing occupied sites in a~\texttt{Rust} hash table.) 
If there is an intersection, the move is immediately rejected. 
Otherwise, the proposed move preserves knot type. 
If it also preserves or decreases the number of lattice edges, it is immediately accepted. 
Otherwise, if it increases the number of edges, it is accepted only with a small fixed probability. 
This method finds very short simple cubic lattice configurations in a few seconds of computer time.\footnote{As an aside, we note that from these minimal (or near-minimal) simple-cubic lattice embeddings, one can compute minimal embeddings on the face-centred and body-centred cubic lattices. 
One can map the simple-cubic lattice to a sublattice of the FCC and BCC lattices, and then minimize using a similar set of BFACF moves \cite{jansevanrensburgBFACFstyleAlgorithmsPolygons2011}.} 

%Then the edges incident to the face are replaced with their complement, providing the result is still a simple closed loop --- 
%Note that one of these moves changes the length of the embedding by \( \pm 2\) lattice edges, while the other leaves the length unchanged, and that the move never changes knot type.
%In a sampler, the probability of success of these moves must be carefully tuned to ensure convergence to the correct distribution. 
%Since we simply wish to arrive a very short embedding of the knot, we can choose transition probabilities to drive the system towards shorter configurations.

\begin{figure}[t]
	\begin{center}
		\includegraphics[width=0.66\textwidth]{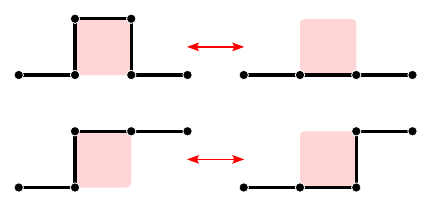}
	\end{center}
	\caption{The deformation moves of the BFACF algorithm on the cubic lattice. Each move acts on an edge and an adjacent face; the edges around that face are replaced by their complement.}
	\label{fig: bfacf1}
\end{figure}

%At each iteration we choose an edge of the embedding and an adjacent face and compute the change in length if the BFACF move is successful. 
%If the move
%leaves the length unchanged or reduces the length, then the move is performed (assuming the result is still a simple closed curve). On the other hand,
%if the move increases the length then we only accept the move with some fixed small probability. In this way the configuration undergoes a biased random
%walk in length, driven towards short lengths, but still able to grow with low probability. We find that the results are quite robust against the biasing
%probabilities chosen. Indeed, we found that a typical initial configuration would very quickly shrink down to within a few edges of the eventual minimum that was achieved.
%Then that minimum was usually found within a few seconds of computer time. We also experimented by tuning the length increase/decrease probabilities so that the length of the configuration could
%walk randomly within a fixed range, with the maximum allowed length slowly ``ratcheted'' down. Again, we found our results to be quite robust against these variations.

\begin{figure}[t]
	\begin{center}
		\includegraphics[width=0.66\textwidth]{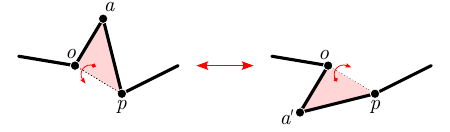}
	\end{center}
	\caption{The fold move for off-lattice polygons. Pick three consecutive vertices giving a triangle of two edges in the polygon and a base axis edge that is not. Then rotate the apex of
		this triangle about the base axis by some angle. % (bottom-left) Pick four consecutive vertices giving a quadrilateral of three edges in the polygon and a base axis that is not. Provided the base-axis
		% 		has length \(\leq 2\) the three edges can be replaces by a pair of edges giving the diagram (bottom-right).
	}
	\label{fig: bfacf2}
\end{figure}

To go off-lattice, we define new moves. The first move (\Cref{fig: bfacf2}) takes a triple of consecutive vertices $o$, $a$, $p$ and rotates $a$ around the line $op$ by a randomly chosen angle to a new position $a'$. 
We imagine that $a$ moved to $a'$ along a line, sweeping out triangles $\triangle oaa'$ and $\triangle paa'$. 
If these triangles do not intersect the remainder of the polygon, this was an ambient isotopy and the knot type is preserved.\footnote{These segment/triangle intersection checks are a standard predicate in computational geometry. They can be done exactly (see~\cite{CGAL}), but we simply choose a conservative floating-point strategy. We might reject some proposed moves where the polygon narrowly avoided a self-intersection this way, but this can only slow things down. We are much more concerned about mistakenly accepting a move which did contain a self-intersection, as this might invalidate the entire computation by changing the knot type of the polygon. We recheck this crucial point at the end of the computation by verifying that the knot type did not change.} In this case, we say the move ``passes triangle checks''.

The second move (see top row in \Cref{fig: bfacf3}) is defined on a set of four consecutive vertices \(o,a,b,p\). 
If $|o-p|>2$, the move is immediately rejected. 
Otherwise, we may choose a plane containing $op$ uniformly at random and construct $\triangle ocp$ in that plane so that $|o-c|=|c-p|=1$. As before, we imagine that $a$ and $b$ move to $c$ along lines, sweeping out triangles \(\triangle oac, \triangle abc, \triangle bcp\). 
This is an ambient isotopy if the rest of the polygon does not intersect these triangles.

The third move (see bottom row in \Cref{fig: bfacf3}) is defined on three consecutive vertices \(o,c,p\). We construct an isosceles trapezoid \(oabp\) in the plane of \(o,c,p\) with $a$, $b$ on the same side of $op$ as $c$ with \(|o-a|=|a-b|=|b-p|=1\). Now we imagine that $c$ splits in two and moves to $a$ and $b$ along lines, sweeping out \(\triangle oac, \triangle abc, \triangle bcp\). 
As before, this is an ambient isotopy if these triangles are disjoint from the rest of the polygon.

We choose one of these moves at random (their relative probabilities seemed to make little difference in the results). 
It is immediately rejected if it fails triangle checks.\footnote{A move which passes triangle checks is guaranteed to preserve knot type. However, triangle checks are conservative: some proposed moves which fail might preserve knot type anyway. We don't check for this.} 
Otherwise, it is immediately accepted if the number of edges stays the same or goes down. 
If the move increases the number of edges, then we accept it only with a small fixed probability.
This procedure is quite effective at reducing the number of edges from the number in the on-lattice configuration, and we generally ran it for only a few minutes. 

\begin{figure}[t]
	\begin{center}
		\includegraphics[width=0.66\textwidth]{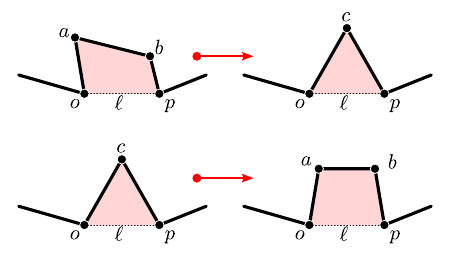}
	\end{center}
	\caption{Top row: The length-decreasing move takes four vertices \(o,a,b,p\) as input; 
	provided the distance \(\ell = |o-p|\) is at most 2, it maps them to three vertices \(o,c,p\). The vertex $c$ is constructed by choosing a unit vector $r$ uniformly from the circle of unit vectors orthogonal to $oc$, and then choosing $co$ to be the unique vector so that $\langle c-o,r \rangle > 0$ and $|o-c| = |c-p| = 1$.   
	Bottom row: The length-increasing move takes three vertices \(o,c,p\) and maps them to four vertices \(o,a,b,p\) so that \(|o-a|=|a-b|=|b-p|=1\) lying in the same plane. }
	\label{fig: bfacf3}
\end{figure}

%Unlike BFACF on the lattice, we must test each of these moves to make sure that topology is preserved. In the case of the length-preserving rotation move,
%vertices \(o,a,p\) become \(o,a',p\). The movement of \(a \mapsto a'\) along a straight line and so maps out two triangular faces \(\triangle oaa', \triangle paa'\)
%To preserve the knot type during this move, we require that no other edge of the polygon intersects either of these triangular faces. That test can be
%performed using M\"oller–Trumbore~\cite{moller2005}. Similarly, when we transform \(o,a,b,p\) to or from \(o,c,p\) we map out three triangles \(\triangle oac, \triangle abc, \triangle bcp\)
%and we check that no other edges of the polygon cross these triangles. These are \(O(n)\) computations. We have not implemented ``clever'' data-structures to speed up these
%checks but instead rely on the fact that we quickly reach configurations of very modest numbers of edges.
%
%Using these moves we run a similar length-reducing algorithm in which a vertex is picked uniformly at random, and then with some probabilities we either
%attempt a length-preserving rotation on the next two edges, a length-decreasing move by replacing the next two edges with three edges, or a length-increasing
%move by replacing the next three edges with two edges. Again, since we are not concerned with sampling from a particular distribution we have considerable freedom
%in choosing these probabilities. In practice we found that our results were fairly robust against variation of these parameters.

So far, all of our moves have preserved the edge lengths of the polygon. In the last stages, we keep the rotation moves (and their triangle checks), but replace the moves which change the number of edges with new ones. The first new move is the ``triangle-collapse'' move, defined on a triple of adjacent vertices \(o,c,p\). It is rejected immediately if \( |o-p|^2 \leq 1/8 \). If not, the proposed move just deletes $c$, replacing $oc$ and $cp$ with $op$. We check $\triangle opc$ for intersections with the rest of the polygon.

The second new move is a ``triangle-inflation'' move, defined on a pair of adjacent vertices \(o,p\). We choose a unit vector $r$ orthogonal to $op$ at random, choose $(t,u)$ uniformly in $[-1/2,3/2] \times [0,1]$ and let \(c = o + t \,(p-o) + u \, r \). We then check $\triangle opc$ for intersections with the rest of the polygon. 

As before, a proposed move is rejected immediately if it does not pass triangle checks. Otherwise, it is accepted if it preserves or reduces the number of edges and accepted with a small fixed probability if it increases the number of edges. This typically reduced stick number by zero, one, or two edges early in the run (within a few seconds) and then seemed to reach minimum stick number. Since we could not be absolutely sure that we had reached minimum stick number unless we matched an existing lower bound (which was quite rare), we usually just continued the run as long as our computational resources allowed. On average, this was about $10$ minutes per knot, but we continued some runs as long as several hours.

We validated the resulting configurations by checking their knot types in two ways. For all knots with $\leq 13$ crossings, we used \pyknotid~\cite{pyknotid} to convert vertex coordinates into a PD code, then used \snappy~\cite{snappy} to identify the knot complement from the manifold census. In all but two cases \snappy\ identified the complement as being the complement of the correct knot;\footnote{Note: we are using the KnotInfo~\cite{knotinfo} naming conventions, so \snappy\ identifies our $10_{83}$ as $10_{86}$ and {\em vice versa}, but this is the correct, and expected, behavior.} the two exceptions were the 13-crossing satellite knots $13n_{4587}$ and $13n_{4639}$, which are not hyperbolic. We also used \knoodle~\cite{knoodle} to create and simplify a planar diagram for each polygon. \knoodle\ contains a database of all possible simplified digrams with $\leq 13$ crossings. Each of these diagrams has been connected by isotopy to one of the planar diagrams given by KnotInfo's PD code for the corresponding knot type (or to a representative PD code which has been mirrored or reversed, depending on the symmetries of the knot type). In this way, \knoodle\ can unambiguously identify knots (including symmetry type) up to 13 crossings.
For all $\num{12965}$ knot types in our experiment, \knoodle\ confirmed that our stick knots were isotopic to~KnotInfo's PD codes.

For knots with more than 13 crossings, we had to use invariants to provide a (weaker) check on knot types. For the $(q,2)$- and $(q,3)$-torus knots and all the twist knots, we verified that the Jones polynomial of the final polygon was the same as the (known) Jones polynomial of that torus or twist knot. For the pretzel and large knots we verified that the Jones polynomial of the polygon was the same as the Jones polynomial of the input diagram. For the $(q,p)$-torus knots with $4 \leq p < q \leq 32$, we used  \knoodle\ to evaluate the Alexander polynomial at $-1$ as well as at 100 random points on the unit circle and compared this to evaluation of the Alexander polynomial of the torus knot at those points. The magnitudes of the actual values were all in the interval $[0.000013, 274.2]$ and the relative errors were all less than $1.53 \times 10^{-9}$.

\section{Results}
\label{sec:results}

Attempting to minimize the number of edges over all polygons with a given knot type $K$ must yield an upper bound for $\stick[K]$;~\cite{schareinInteractiveTopologicalDrawing1998,rawdonUpperBoundsEquilateral2002} have given stick number bounds for all knot types with 10 and fewer crossings in this way. Sampling $n$-gons only provides an upper bound for the knot types observed in the sample; this provided stick number bounds for many 11 crossing knot types and some knot types with 12 and 13 crossings~\cite{eddyNewStickNumber2022, stick-knot-gen}. As we ran examples for every knot type up through 13 crossing, we are able to fill in the missing data in this table. In every case, we were able to find configurations which matched the existing bounds, and for $\num{12160}$ of $\num{12965}$ knot types, we were able to improve them.

The new stick number bounds for knot types up to 10 crossings appear in~\Cref{tab:new stick bounds}, while the full table of new stick number bounds is distributed with this paper as supplementary data. The polygons can be accessed as a dataset on the Harvard Dataverse~\cite{dataverse-table}. Frequencies of different stick numbers are reported in~\Cref{tab:frequencies}. Many of the configurations we found appear (almost) singular to the eye (see, e.g., \Cref{fig:10 stick exact,fig:9_36 and 10_37,fig:9_29}), but they are actually several orders of magnitude away from being numerically singular: the minimum distance between non-adjacent edges in any of the knots we found was $4.18148 \times  10^{-7}$ (between the 10th and 12th edges of our $13n_{2931}$).

\begin{table}
	\begin{center}
		\begin{tabular}{l@{\hskip 0.25in}llllllll}
			\multicolumn{9}{c}{Stick number frequencies}                       \\
			\toprule
			\diagbox{$c$}{$\stick$} & 6 & 7 & 8 & 9  & 10  & 11   & 12   & 13  \\
			\midrule
			3                       & 1 &   &   &    &     &      &      &     \\
			4                       &   & 1 &   &    &     &      &      &     \\
			5                       &   &   & 2 &    &     &      &      &     \\
			6                       &   &   & 3 &    &     &      &      &     \\
			7                       &   &   &   & 7  &     &      &      &     \\
			8                       &   &   & 2 & 4  & 15  &      &      &     \\
			9                       &   &   &   & 14 & 35  &      &      &     \\
			10                      &   &   &   &    & 94  & 71   &      &     \\
			11                      &   &   &   &    & 160 & 362  & 30   &     \\
			12                      &   &   &   &    & 122 & 1156 & 898  &     \\
			13                      &   &   &   &    & 108 & 2397 & 6509 & 974 \\
			\bottomrule
		\end{tabular}
	\end{center}
	\caption{Frequencies of different stick numbers, by crossing number.}
	\label{tab:frequencies}
\end{table}

Our new upper bounds allow us to compute exact stick numbers for 19 knot types:
\begin{sticknumber}
	The following knots have stick number equal to 10:
	\begin{itemize}
		\item $11n_i$ with $i \in \{71,75,76,78\}$;
		\item $13n_j$ with $j \in \{225, 230, 285, 288, 307, 584, 586, 593\}$, \\
		      and $j \in \{602, 603, 604, 607, 608, 1192, 5018\}$.
	\end{itemize}
\end{sticknumber}

\begin{proof}[Proof of \Cref{thm:exact-sticks}]
	We found 10-stick realizations of all of the knots in the statement of the theorem, which implies $\stick[K] \leq 10$ for each of these knots.

	On the other hand, each of the knots in the statement of the theorem is a 4-bridge knot: for the 11-crossing knots this was proved by Musick~\cite{musickMinimalBridgeProjections2012}, and for the 13-crossing knots by Blair, Kjuchukova, and Morrison~\cite{blairCoxeterQuotientsKnot2025,Wirt_Hm}. Work of Kuiper~\cite{kuiperNewKnotInvariant1987} and Jin~\cite{jinPolygonIndicesSuperbridge1997} implies that
	\begin{equation}\label{eq:bridge-superbridge-stick inequality}
		\bridge[K] < \superbridge[K] \leq \frac{1}{2} \stick[K],
	\end{equation}
	where $\bridge[K]$ is the bridge index of a knot $K$ and $\superbridge[K]$ is its superbridge index. Hence, every 4-bridge knot has $\stick[K] \geq 10$. 

	Since we've proved that $\stick[K]$ is bounded above and below by 10, we conclude that it must be exactly 10 for each of these knots. 
\end{proof}

We are now in a position to prove our theorem on the asymptotic relationship between stick number and crossing number, which we restate for convenience. 
\begin{asymptotics}
	There are constants $\alpha_*$ and $\beta_*$ so that for any $\alpha < \alpha_*$ and $\beta > \beta_*$ there is some $C$ so that for all knot types with $c[K] > C$, we have
	\[
		\alpha \sqrt{c[K]} \leq \stick[K] \leq \beta \, c[K].
	\]
	Further, $\alpha_* \in [\!\sqrt{2},2]$ while $\beta_* \in [\nicefrac{2}{3}, \nicefrac{3}{2}]$.
\end{asymptotics}

\begin{proof}[Proof of \Cref{thm:asymptotic picture}]
	In 1997, Adams et al.~\cite{adamsStickNumbersComposition1997} proved that the $(p+1,p)$-torus knot has $\stick[T_{p+1,p}] = 2p+2$. Since $c[T_{p+1,p}] = p^2-1$, this implies that
	\begin{equation}\label{eq:min torus stick}
		\stick[T_{p+1,p}] = 2\sqrt{c[T_{p+1,p}]+1}+2.
	\end{equation}
	Hence, $\alpha_*$ is finite and $\leq 2$, though this leaves open the possibility that $\alpha_*=0$.

	This is ruled out by the general bound proved by Calvo~\cite{calvoGeometricKnotSpaces2001} (slightly improving an earlier bound by Negami~\cite{negamiRamseyTheoremsKnots1991}): for all knots $K$,
	\[
		\stick[K] \geq \sqrt{2}\sqrt{c[K]+\nicefrac{1}{8}} + \frac{7}{2}.
	\]
	Hence, $\alpha_* \geq \sqrt{2}$ and, combining with \eqref{eq:min torus stick}, we conclude that $\alpha_* \in [\!\sqrt{2},2]$.

	Turning to the upper bound on stick number, the best general bound was proved by Huh and Oh~\cite{huhUpperBoundStick2011}: for all nontrivial knots $K$,
	\begin{equation}\label{eq:general upper bound}
		\stick[K] \leq \frac{3}{2}c[K]+\frac{3}{2}.
	\end{equation}
	Therefore, $\beta_* \leq \nicefrac{3}{2}$.

	\begin{figure}[htbp]
		\centering
		\includegraphics[height=1.5in]{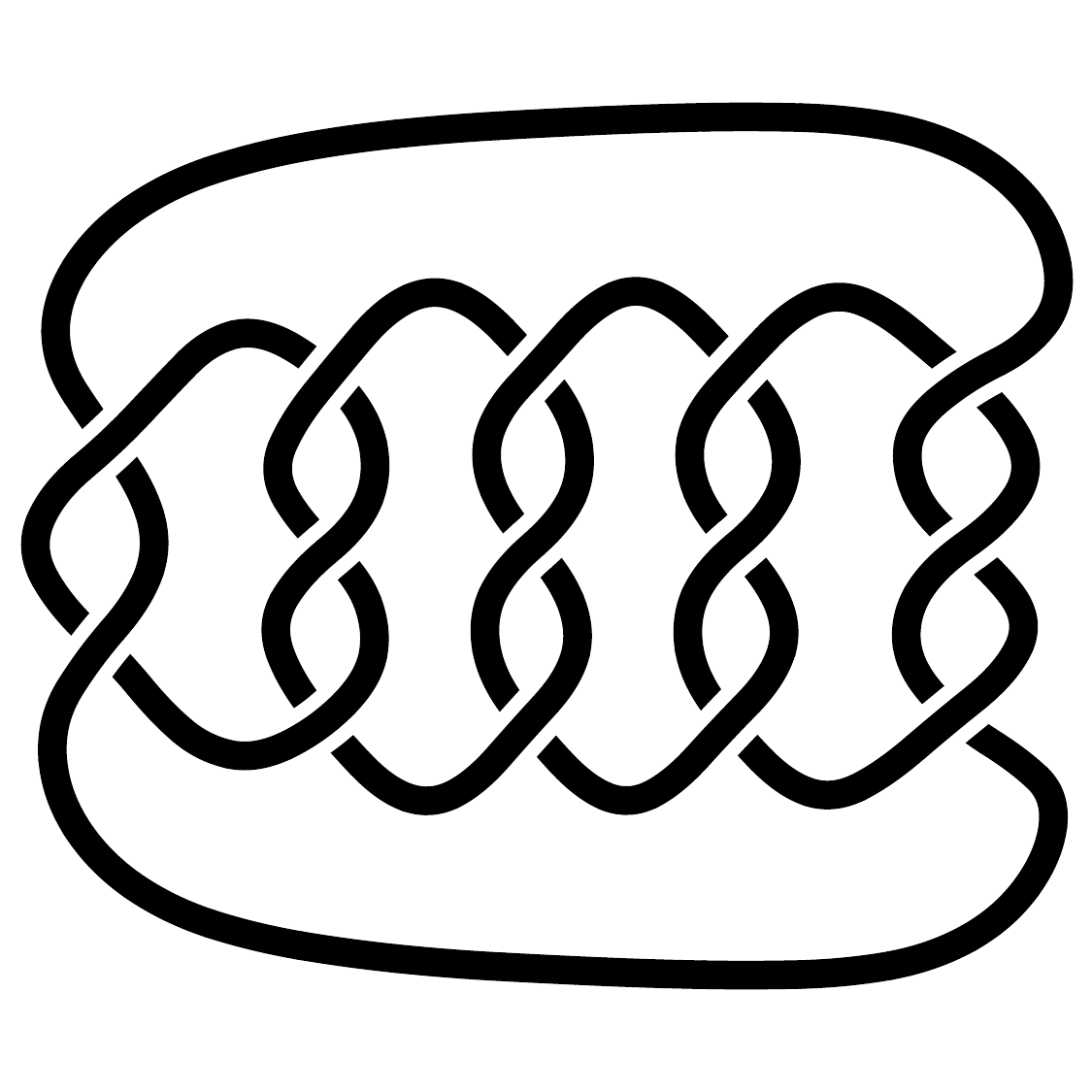}
		\caption{The pretzel knot $P_4 = P(2,3,3,3,3)$.}
		\label{fig:pretzel}
	\end{figure}

	To rule out the possibility that $\beta_* = 0$, we now construct an infinite family of knots whose stick numbers are linear in their crossing numbers. 
	Consider the pretzel knots $P_m \ceq P(2,3,\dots , 3)$, where $m$ is the number of $3$s. 
	An alternating diagram for $P_4$ is shown in \Cref{fig:pretzel}; there are similar diagrams for all $P_m$.  
	Since these diagrams are reduced and alternating, they are minimal crossing number diagrams~\cite{MR899057} showing that $c[P_m] = 3m+2$. 
	
	On the other hand, the bridge index $\bridge[P_m] = m+1$~\cite{baaderCoxeterGroupsMeridional2021,boileauNombrePontsGenerateurs1985}, so \eqref{eq:bridge-superbridge-stick inequality} implies that
	\[
		\stick[P_m] \geq 2\superbridge[P_m] \geq 2 (\bridge[P_m] + 1) = \frac{2}{3}c[P_m] + \frac{8}{3}.
	\]
	Therefore, $\beta_* \geq \nicefrac{2}{3}$; combining this with the Huh–Oh upper bound \eqref{eq:general upper bound} yields the result $\beta_* \in [\nicefrac{2}{3},\nicefrac{3}{2}]$.
\end{proof}

We observe that for all the $12$- and $13$-crossing knots, we were able to show that  $\stick[K] \leq c[K]$, suggesting that $\beta_* \leq 1$. (This is known for $2$-bridge knots by~\cite{huhStickNumbers2bridge2011}.)
% Curious about whether this persists for more complicated knots, we tried the 32 and 64-crossing alternating and nonalternating knots shown in~\autoref{fig:complicated examples}. We were able to show $\stick[K] \leq c[K]$ in all four cases.

\section{Discussion and Future Directions}
\label{sec:discussion}

Our algorithm always matched or improved on existing bounds for stick numbers of knots through 13 crossings. To test it further, we compared it to the paper~\cite{johnsonStickRamseyNumbers2013}, which gives a handcrafted construction for polygonal representatives of the $(q,2)$-torus knots with $\left\lfloor \frac{2}{3}(q+1) \right\rfloor + 4$ edges. Since $c[T_{q,2}] = q$, 
this implies that $\stick[T_{q,2}] \lesssim \frac{2}{3}c[T_{q,2}]$. We tried $3 \leq q \leq 31$, finding a configuration with stick number equal to the bound above in each case. Interestingly, we were unable to improve on this result, suggesting that it may be optimal.

We also tried $T_{q,3}$ torus knots, for $4 \leq q \leq 26$, obtaining the data in~\autoref{tab:torus knots 2}. This data is clearly highly structured, and suggests that $\stick[T_{q,3}] \lesssim\frac{5}{16} c[T_{q,3}]$ for large $q$, but we don't have an explicit conjectured formula;
compare with the known result $\stick[T_{q,3}] \lesssim \frac{1}{3} c[T_{q,3}]$ from~\cite{johnsonStickRamseyNumbers2013}.

We provide polygonal configurations for these torus knots (plus all other torus knots with $2\leq p < q \leq 32$) in the dataset~\cite{dataverse-torus}.

\begin{table}[ht]
	\label{tab:torus knots 2}
	\begin{tabular}{c@{\hskip 0.25in}cccccccccccccccccccccc}
		\toprule
		$q$                    & 4 & 5  & 7  & 8  & 10 & 11 & 13 & 14 & 16 & 17 & 19 & 20 & 22 & 23 & 25 & 26 \\
		\midrule
		$c[T_{q,3}]$           & 8 & 10 & 14 & 16 & 20 & 22 & 26 & 28 & 32 & 34 & 38 & 40 & 44 & 46 & 50 & 52 \\
		$\stick[T_{q,3}] \leq$ & 8 & 10 & 12 & 12 & 13 & 14 & 15 & 16 & 18 & 18 & 19 & 20 & 21 & 22 & 23 & 24 \\
		\bottomrule
	\end{tabular}
	\caption{Computed bounds for stick numbers of $T_{q,3}$ torus knots. It would be interesting to find a formula or systematic construction of polygons that matches this data.}
\end{table}

%In this case we find
%\[
%	\stick(P_m) \leq \begin{cases}
%		8    & m = 1            \\
%		10   & m= 2             \\
%		3m+2 & 3 \leq m \leq 10. \\
%	\end{cases}
%\]
%This pattern continues excepting a few places where we get \(3m+1\) as the bound:
%\[
%	\begin{array}{|c||c|c|c|c|c|c|c|c|c|c|c|c|c|c|c|c|c|c|c|c|c|c|}
%		\hline
%		q                & 11 & 12 & 13 & 14 & 15 & 16 & 17 & 18 & 19 & 20 & 21 \\
%		\stick(P_q) \leq & 34 & 38 & 40 & 44 & 46 & 50 & 52 & 55 & 59 & 62 & 65 \\
%		\hline
%		q                & 22 & 23 & 24 & 25 & 26 & 27 & 28 & 29 & 30 & 31 & 32 \\
%		\stick(P_q) \leq & 68 & 71 & 74 & 77 & 80 & 83 & 86 & *  & *  & *  & *  \\
%		\hline
%	\end{array}
%\]
%FIX THE ASTERIX; WRITE THE TABLE IN TERMS OF CROSSING NUMBER

We next constructed lattice embeddings of twist knots \(K_p\) for a range of \(p\)-values and then minimised them through our BFACF-style algorithms, obtaining the data in~\autoref{tab:twist knots}. The polygonal configurations we found are available in the dataset~\cite{dataverse-twist}.
\begin{table}[ht]
	\centering
	\label{tab:twist knots}
	\begin{tabular}{c@{\hskip 0.25in}cccccccccc}
		\toprule
		$p$                & 1  & 2  & 3  & 4  & 5  & 6  & 7  & 8  & 9  & 10 \\
		\midrule
		$c[K_p]$           & 3  & 4  & 5  & 6  & 7  & 8  & 9  & 10 & 11 & 12
		\\
		$\stick[K_p] \leq$ & 6  & 7  & 8  & 8  & 9  & 10 & 10 & 11 & 12 & 12 \\
		\toprule
		$p$                & 11 & 12 & 13 & 14 & 15 & 16 & 17 & 18 & 19 & 20
		\\		\midrule
		$c[K_p]$           & 13 & 14 & 15 & 16 & 17 & 18 & 19 & 20 & 21 & 22
		\\
		$\stick[K_p] \leq$ & 13 & 14 & 14 & 15 & 16 & 16 & 17 & 18 & 18 & 19 \\
		\bottomrule
	\end{tabular}
	\caption{Stick number bounds obtained by our algorithm for the twist knot $K_p$ where $p$ is the number of half-twists. The data is consistent with the conjecture $\stick[K_p] \leq \left\lfloor \frac{2p}{3} \right\rfloor + 6$.}
\end{table}
%This is consistent with the following conjectured bound:
%\begin{conjecture}
%	The stick number of the twist knot \(K_p\) is bounded by
%	\[
%		\stick[K_p] \leq \left\lfloor \frac{2p}{3} \right\rfloor + 6.
%	\]
%\end{conjecture}
%Since $c[K_p] = p+2$, this would imply $\stick[K_p] \lesssim \frac{2}{3}c[K_p]$.

For these torus and twist knots, we found configurations with many fewer sticks than the crossing number of the knot. On the other hand, we tried to find low-stick configurations of the pretzel knots $P_m$ from the proof of \autoref{thm:asymptotic picture} for $3 \leq m \leq 28$. With some modest effort, we could not reduce the number of edges below $c[P_m]-1$; see the polygon coordinates at~\cite{dataverse-pretzel}. This (weakly) suggests that $\beta_* \geq 1$ and makes it seem interesting to try to find polygonal versions of these knots with fewer sticks. We hope that someone will try.

\begin{figure}[ht]
	\label{fig:complicated examples}
	\centering
	\subfloat[{\begin{minipage}{.4\textwidth}
					\centering
					$28 \leq c[K] \leq 32$\qquad$\stick[K] \leq 25$
				\end{minipage}}]{\includegraphics[width=.2\textwidth,height=.2\textwidth,keepaspectratio,valign=c]{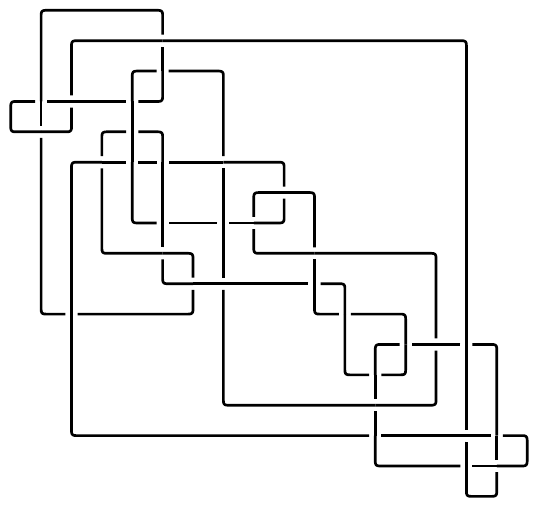}\quad\includegraphics[width=.2\textwidth,height=.2\textwidth,keepaspectratio,valign=c]{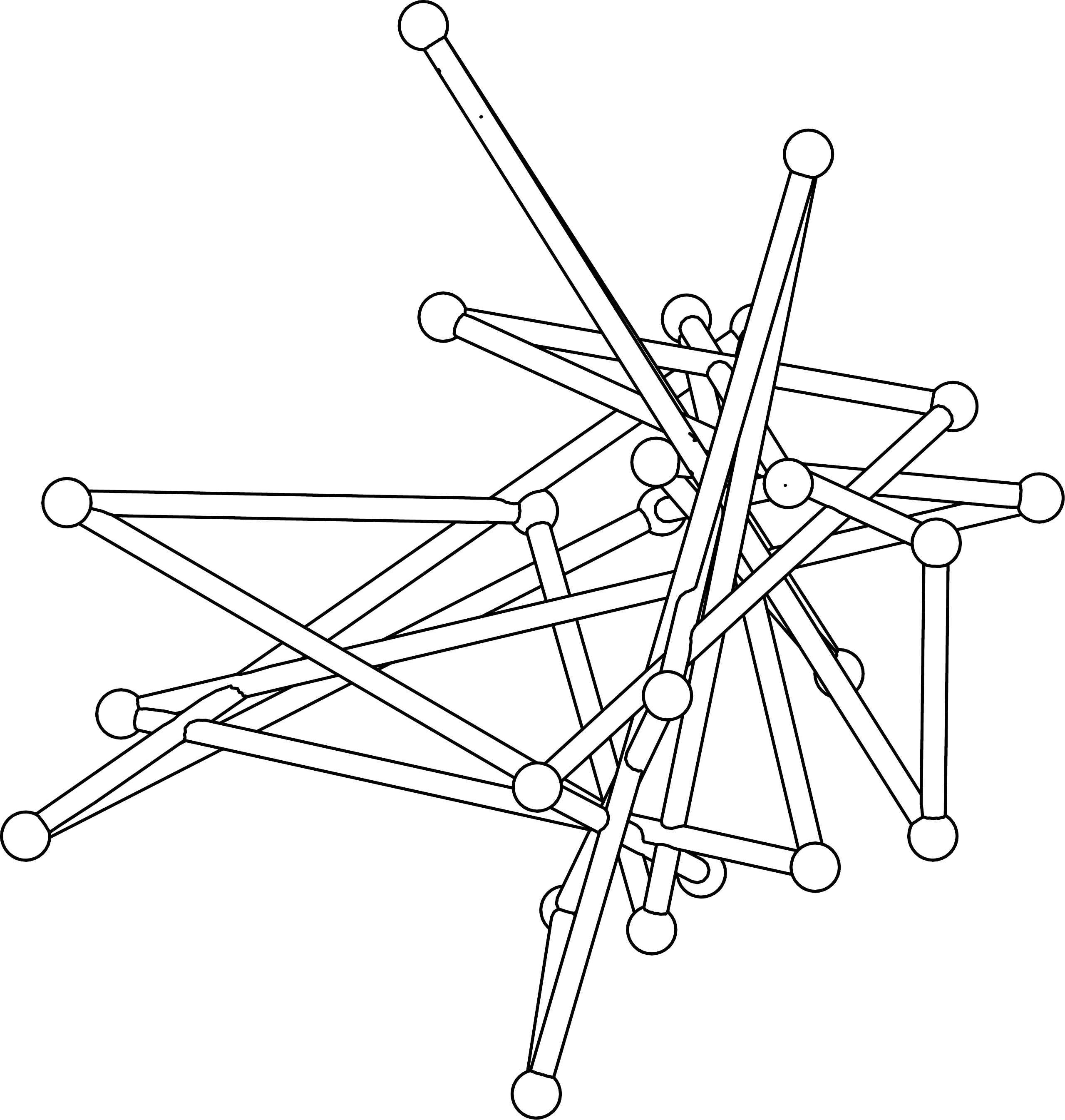}}
	\qquad \qquad
	\subfloat[{\begin{minipage}{.4\textwidth}
					\centering
					$46 \leq c[K] \leq 64$\qquad$\stick[K] \leq 40$
				\end{minipage}}]{\includegraphics[width=.2\textwidth,height=.2\textwidth,keepaspectratio,valign=c]{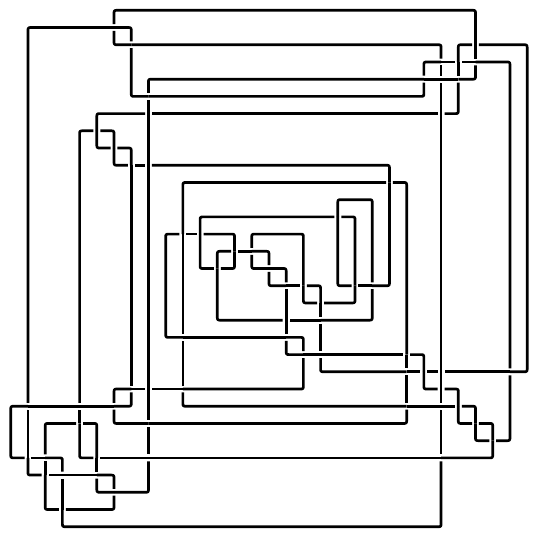}\quad \includegraphics[width=.2\textwidth,height=.2\textwidth,keepaspectratio,valign=c]{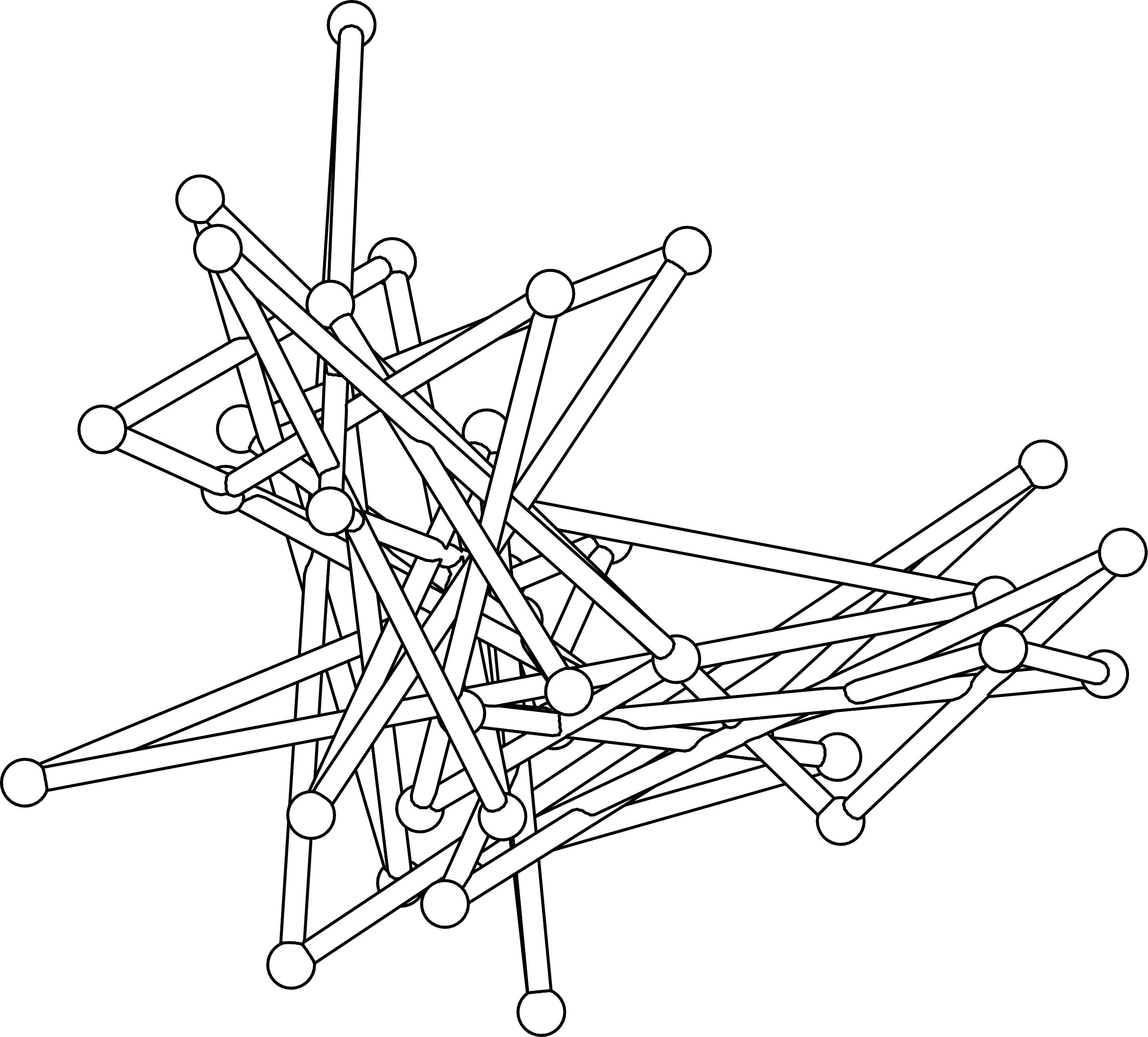}}\\
	\vspace{.2in}
	\subfloat[{\begin{minipage}{.4\textwidth}
					\centering
					$c[K] = 32$ \qquad$\stick[K] \leq 23$
				\end{minipage}}]{\includegraphics[width=.2\textwidth,height=.2\textwidth,keepaspectratio,valign=c]{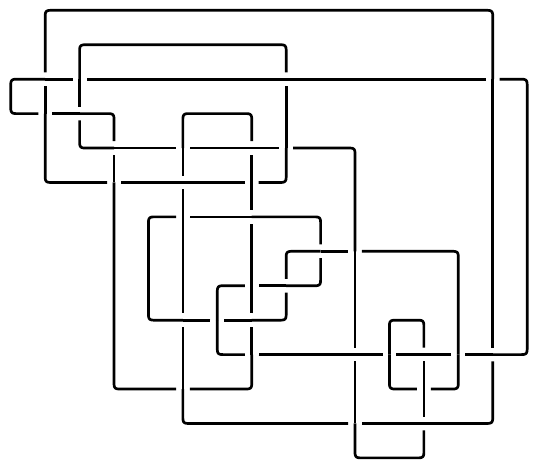}\quad \includegraphics[width=.2\textwidth,height=.2\textwidth,keepaspectratio,valign=c]{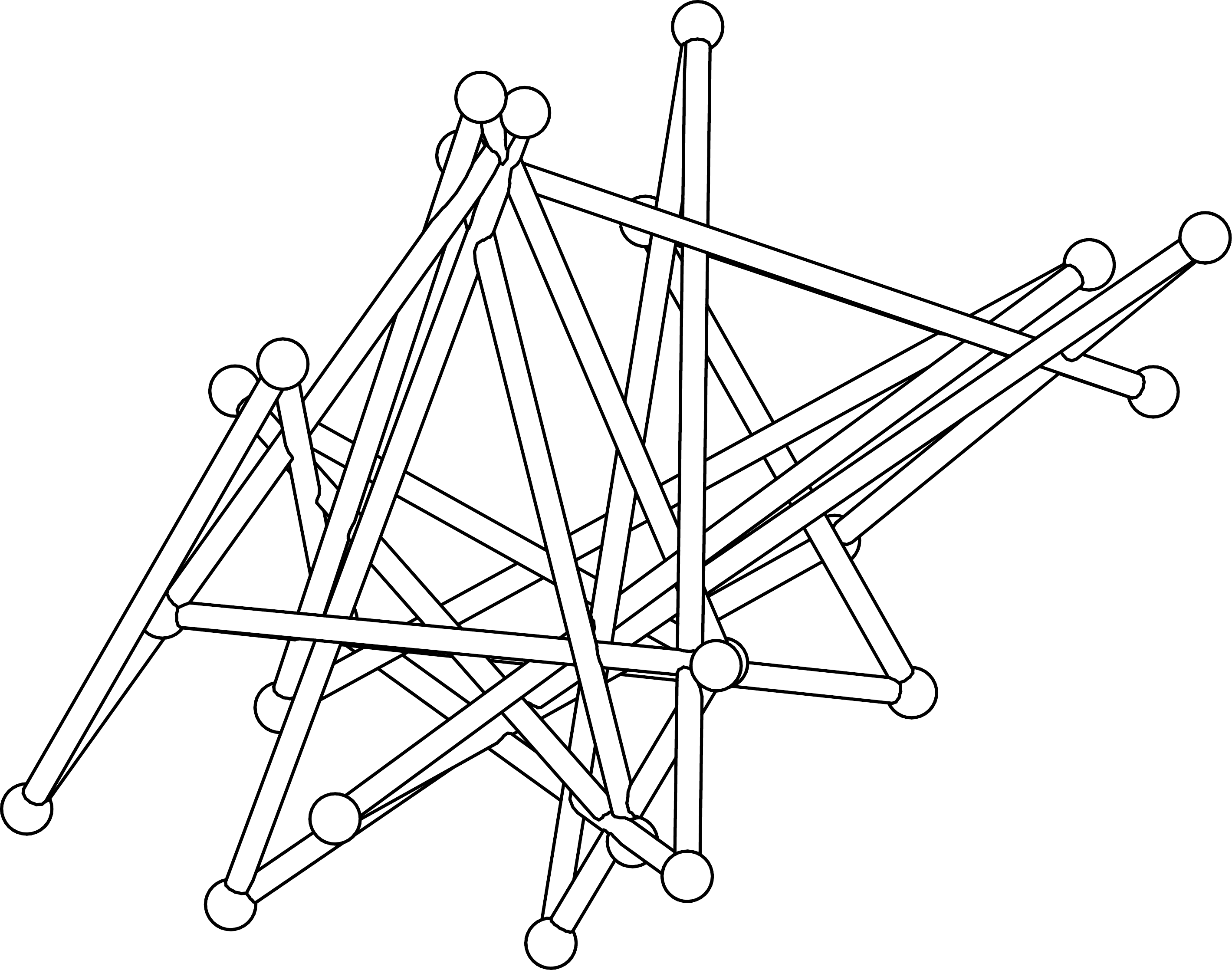}}
	\qquad \qquad
	\subfloat[{\begin{minipage}{.4\textwidth}
					\centering
					$c[K] = 64$ \qquad$\stick[K] \leq 48$
				\end{minipage}}]{\includegraphics[width=.2\textwidth,height=.2\textwidth,keepaspectratio,valign=c]{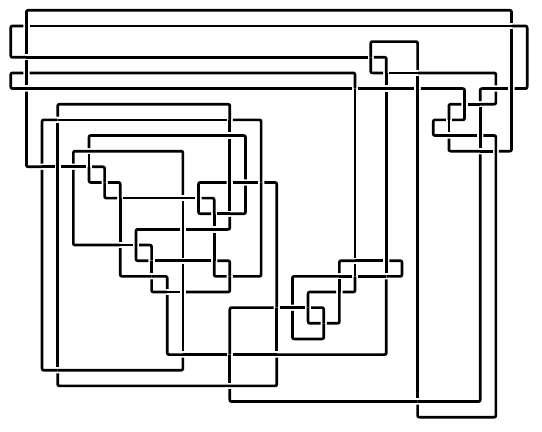}\quad \includegraphics[width=.2\textwidth,height=.2\textwidth,keepaspectratio,valign=c]{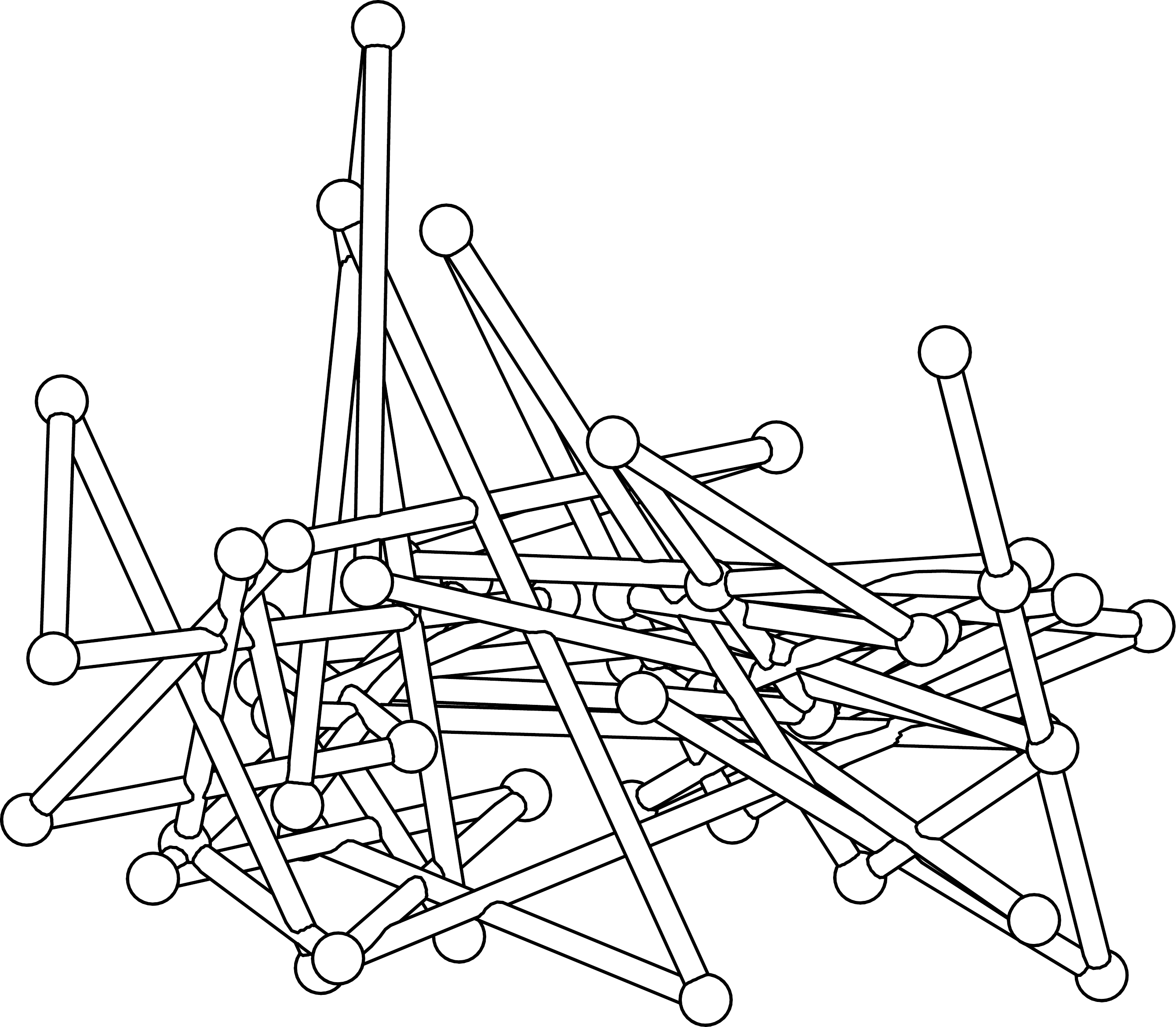}}
	\caption{We tried four large, untabulated knots to see if their stick number could be bounded above by crossing number.
		The non-alternating knots in the top row were prime summands of equilateral random polygons with a few thousand edges.
		The alternating knots on the bottom row were produced by changing crossings in diagrams generated the same way.
		All four of the resulting knots are hyperbolic, and therefore prime. The alternating diagrams are reduced, and hence the
		crossing number of the diagram is equal to the crossing number of the knot. We believe the non-alternating diagrams
		are minimal or nearly so, but can't prove this; the lower bound on crossing number comes from the span of the Jones
		polynomial. We also believe the stick embeddings are not far from minimal. Note that in each case the bound on
		stick number is smaller than the crossing number.}
\end{figure}

We also generated some large random (untabulated) knots of up to 64 crossings by taking prime summands of equilateral random polygons with a few thousand edges. See the diagrams in \autoref{fig:complicated examples} (or larger in \autoref{fig:complicated examples 1}) and the underlying data at~\cite{dataverse-big}. From each diagram we constructed
a lattice embedding (see \autoref{fig:complicated examples 2}) using \knoodle's \texttt{REAPR} embedding. These were far from minimal and they were reduced significantly using the BFACF algorithm as described in \Cref{sec:methods}. This was perhaps
the slowest step in bounding the stick-number; letting the very flat initial embeddings ``relax'' sometimes required allowing them to grow to almost twice the initial length before they reduced down to far more manageable sizes.
Once much shorter lattice embeddings were found (see \autoref{fig:complicated examples 3}), we reduced them first with unit-sticks (\autoref{fig:complicated examples 4}) and
finally sticks of any length (see \autoref{fig:complicated examples} or larger in \autoref{fig:complicated examples 5}). In each case we were able to reduce the configuration down to one that required fewer sticks than its crossing number.

Since we have not been able to find a counterexample, we propose the following conjecture, which is slightly stronger than the conjecture that $\beta_* \leq 1$:

\begin{conjecture}\label{conj:stick bound}
	For all knots $K$ with $c[K] \geq 12$,
	\[
		\stick[K] \leq c[K].
	\]
\end{conjecture}

In general, it would be interesting to know the distribution of $\stick[K]$ over all prime knots with a given number of crossings. \Cref{thm:asymptotic picture} gives quite a broad range, but is the distribution weighted towards one end or the other? In principle, given Burton and Thistlethwaite's classifications~\cite{burtonNext350Million2020,thistlethwaiteEnumerationClassificationPrime2025} and sufficient computer time, one could collect data on this distribution for all prime knot types through 20 crossings.

We close with one last observation. As described in~\Cref{sec:methods}, we produced equilateral configurations of each knot type as an intermediate step to finding our minimal stick configurations. We have not tried too hard to optimize these equilateral configurations, but it remains an open question whether \emph{equilateral stick number}---that is, the minimum number of segments in any polygonal realization in which all segments have the same length---is a distinct invariant from stick number. In 2002 Rawdon and Scharein~\cite{rawdonUpperBoundsEquilateral2002} produced a list of potential examples of knots for which these invariants differ. That list has been whittled away over the years~\cite{millettPhysicalKnotTheory2012,eddyNewStickNumber2022,shonkwilerAllPrimeKnots2022}; the only remaining example of a knot with $\leq 10$ crossings for which the best known bounds on stick number and equilateral stick number differ is $9_{29}$, which has $\stick[9_{29}] = 9$~\cite{calvoGeometricKnotTheory1998,schareinInteractiveTopologicalDrawing1998}. We can now prove that the equilateral stick number is also 9:

\begin{theorem}\label{thm:eq 9_29}
	The knot $9_{29}$ has equilateral stick number equal to 9.
\end{theorem}

\begin{figure}[t]
	\centering
	\includegraphics[height=2in]{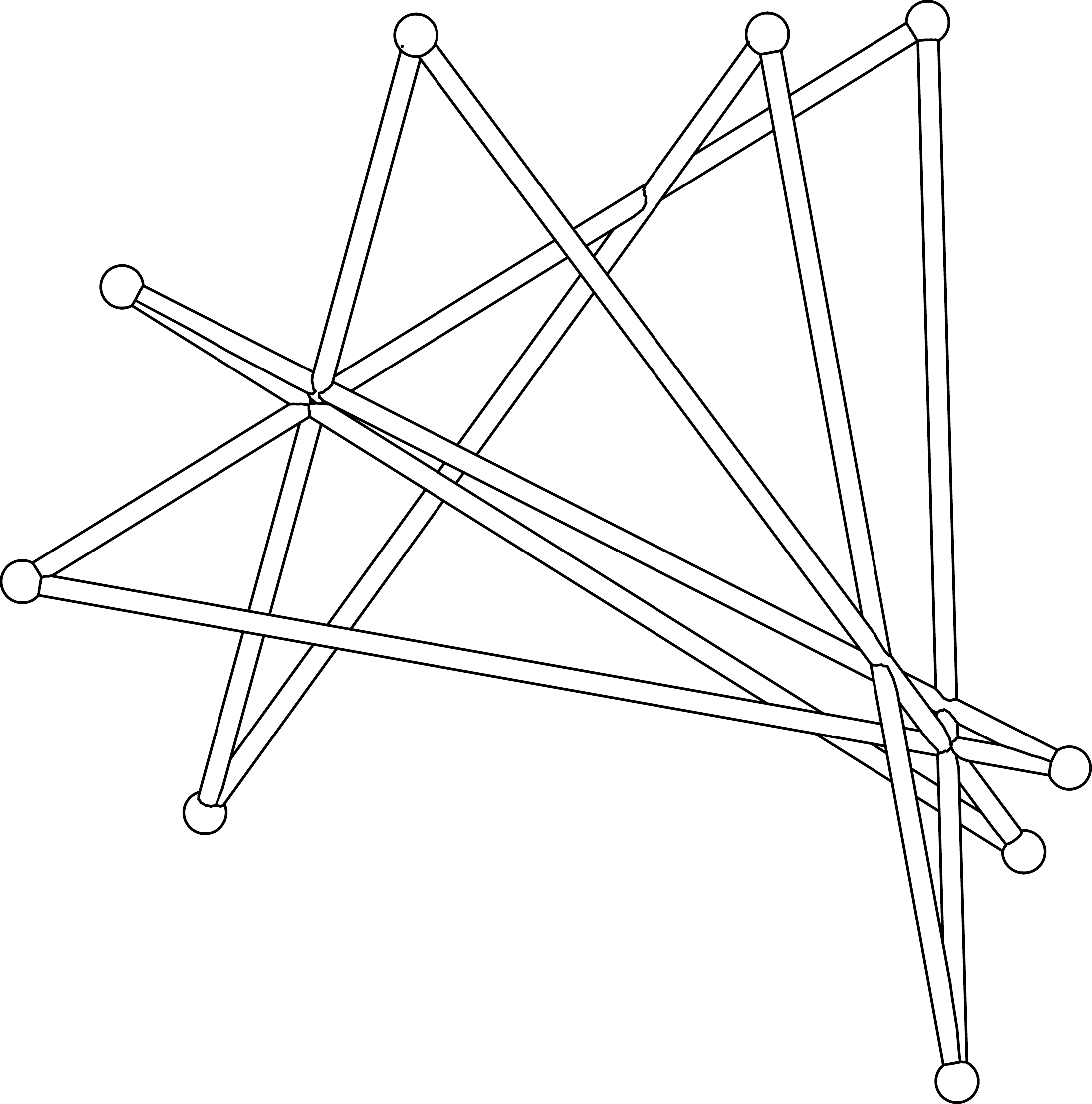}

	\vspace{.3in}

	\sisetup{table-format = 2.18, table-alignment-mode = format}

	\begin{footnotesize}
		\begin{tabular}{SSS}
			-0.01705188882962603  & -0.8472070769441673  & 0.5309890788547935   \\
			0.041386237542474724  & 0.04265321980605685  & 0.07851406556265417  \\
			-0.006817402438744069 & -0.9406813617024611  & 0.25381241096850504  \\
			-0.028582515259995032 & -0.09019744875337588 & 0.7793629734063082   \\
			0.06629662310464379   & -0.06256778735228902 & -0.21574232298902563 \\
			0.06914020976103849   & -0.24079725362023752 & 0.7682425200694282   \\
			-0.3617541314019883   & -0.7692376495316422  & 0.03675033643198222  \\
			0.29177162736175377   & -0.5711160126667256  & 0.7672651546514633   \\
			0.                    & 0.                   & 0.                   \\
		\end{tabular}
	\end{footnotesize}
	\caption{Our equilateral 9-stick $9_{29}$ shown in orthographic perspective, viewed from the direction of the positive $x$-axis relative to the vertex coordinates given below the visualization. These coordinates are also given in the dataset~\cite{dataverse-table}.}
	\label{fig:9_29}
\end{figure}

\begin{proof}
	Since $\stick[9_{29}] = 9$ and equilateral stick number must be at least as big, it suffices to find a 9-stick realization of $9_{29}$ with equal-length edges. Such a realization, together with its vertex coordinates, is shown in \autoref{fig:9_29}. This polygon was found in two steps. First, we took the (non-equilateral) output of our simulated annealer and minimized the deviation from unit-length edges without changing knot type by repeatedly trying to replace a triangle formed by consecutive edges by a triangle with at least one unit-length edge. This eventually produced a 9-edge $9_{29}$ with all edge lengths in the interval $[1, 1.00000455]$. We can think of this polygon as a set of $9$ edge displacement vectors, all with lengths very close to 1, which sum to zero. Rescaling these vectors to length exactly 1 (in double precision) gives us a numerically equilateral polygon which fails to close by a small amount. Applying conformal barycenter closure~\cite{cantarellaComputingConformalBarycenter2022} gives us a closed polygon which is numerically equilateral; we checked using \knoodle\ that this polygon is still $9_{29}$.

	Of course, this realization is only approximately equilateral; to prove there is a true equilateral $9_{29}$ we use a result of Millett and Rawdon~\cite{millettEnergyRopelengthOther2003}, which says that, for an $n$-edge polygonal realization $P$ of a knot type $K$, there is a nearby equilateral realization of $K$ with unit-length edges if, for all $i=1, \dots , n$,
	\begin{equation}\label{eq:millett rawdon}
		|L_i - 1| < \min\left\{\frac{\mu(P)}{n},\frac{\mu(P)^2}{4}\right\}.
	\end{equation}
	Here $L_i$ is the length of the $i$th edge and $\mu(P)$ is the minimum distance between any two non-adjacent edges of $P$.

	Although our $9_{29}$ looks singular to the eye, the minimum distance between non-adjacent edges is $1.84536 \times 10^{-4}$ (between the fourth and eighth edges), whereas the maximum deviation from unit length is $2.22045 \times 10^{-16}$ (for both the first and ninth edges). Hence, for our $9_{29}$ we have
	\[
		|L_i - 1| < 10^{-7.58}\min\left\{\frac{\mu(P)}{9},\frac{\mu(P)^2}{4}\right\}
	\]
	for all $i=1,\dots , 9$, so it easily satisfies the Millett–Rawdon condition~\eqref{eq:millett rawdon}.
\end{proof}

It seems unlikely that stick number and equilateral stick number are always the same, but it remains a substantial challenge to find a specific knot on which they differ.

\section*{Acknowledgments}
We are grateful to Se-Goo Kim and to Chuck Livingston and Allison Moore at KnotInfo~\cite{knotinfo} for providing 3D coordinates of all knots through 13 crossings, and to Robert Lipshitz for his Blender tutorials~\cite{lipshitz-blender}. This paper was inspired by conversations at the Banff International Research Station workshop on \emph{Knot Theory Informed by Random Models and Experimental Data} in April, 2024, so we are very grateful to the workshop organizers and to BIRS for catalyzing this project. We are also very grateful to BIRS for hosting us in July, 2025 for the Research in Teams workshop on \emph{Stick Numbers and Polygonal Knot Theory}, where much of the work on this paper was done. In addition, we would like to acknowledge the generous support of NSERC, the National Science Foundation (DMS--2107700 to Shonkwiler), and the Simons Foundation (\#524120 to Cantarella).

\clearpage

\appendix

\section{Stick Number Bounds for Prime Knots Through 10 Crossings}
\label{sec:rolfsen table}

We give here the number of edges in the configurations we found for all knots in the Rolfsen table, which gives an upper bound on the stick number.
For knots with \known{boxed} edge count, $\stick[K]$ is known exactly and in each case our bound matches this known value.
For knots with \newbound{bracketed} edge count, our bound beats the best previous upper bound on $\stick[K]$.
For the remaining knots, $\stick[K]$ is not known but our examples match the best existing bounds.

\begin{center}
	\begin{multicols*}{4}
		\TrickSupertabularIntoMulticols

		\tablefirsthead{
			$K$ & edges \\
			\midrule
		}
		\tablehead{
			$K$ & edges \\
			\midrule
		}
		\tablelasttail{\midrule}

		\setlength{\tabcolsep}{5pt}

		\begin{supertabular*}{.14\textwidth}{lc}
			\label{tab:new stick bounds}
			\hspace{-.047in}$3_1$ & \known{$6$} \\[1ex]
			\hline\\[-2ex]
			$4_1$ & \known{$7$} \\[1ex]
			\hline\\[-2ex]
			$5_1$ & \known{$8$} \\ \shrinkheight{9pt}
			$5_2$ & \known{$8$} \\[1ex]
			\hline\\[-2ex]
			$6_1$ & \known{$8$} \\
			$6_2$ & \known{$8$} \\
			$6_3$ & \known{$8$} \\[1ex]
			\hline\\[-2ex]
			$7_1$ & \known{$9$} \\
			$7_2$ & \known{$9$} \\
			$7_3$ & \known{$9$} \\
			$7_4$ & \known{$9$} \\
			$7_5$ & \known{$9$} \\
			$7_6$ & \known{$9$} \\
			$7_7$ & \known{$9$} \\[1ex]
			\hline\\[-2ex]
			$8_1$ & $10$ \\
			$8_2$ & $10$ \\
			$8_3$ & $10$ \\
			$8_4$ & $10$ \\
			$8_5$ & $10$ \\
			$8_6$ & $10$ \\
			$8_7$ & $10$ \\
			$8_8$ & $10$ \\
			$8_9$ & $10$ \\
			$8_{10}$ & $10$ \\
			$8_{11}$ & $10$ \\
			$8_{12}$ & $10$ \\
			$8_{13}$ & $10$ \\
			$8_{14}$ & $10$ \\ \shrinkheight{9pt}
			$8_{15}$ & $10$ \\
			$8_{16}$ & \known{$9$} \\
			$8_{17}$ & \known{$9$} \\
			$8_{18}$ & \known{$9$} \\
			$8_{19}$ & \known{$8$} \\
			$8_{20}$ & \known{$8$} \\
			$8_{21}$ & \known{$9$} \\[1ex]
			\hline \\[-2ex]
			$9_1$ & $10$ \\
			$9_2$ & $10$ \\
			$9_3$ & $10$ \\
			$9_4$ & $10$ \\
			$9_5$ & $10$ \\
			$9_6$ & \newbound{$10$} \\
			$9_7$ & $10$ \\
			$9_8$ & $10$ \\
			$9_9$ & $10$ \\
			$9_{10}$ & $10$ \\
			$9_{11}$ & $10$ \\
			$9_{12}$ & $10$ \\
			$9_{13}$ & $10$ \\
			$9_{14}$ & $10$ \\
			$9_{15}$ & $10$ \\ \shrinkheight{9pt}
			$9_{16}$ & $10$ \\
			$9_{17}$ & $10$ \\
			$9_{18}$ & $10$ \\
			$9_{19}$ & $10$ \\
			$9_{20}$ & $10$ \\
			$9_{21}$ & $10$ \\
			$9_{22}$ & $10$ \\
			$9_{23}$ & \newbound{$10$} \\
			$9_{24}$ & $10$ \\
			$9_{25}$ & $10$ \\
			$9_{26}$ & $10$ \\
			$9_{27}$ & $10$ \\
			$9_{28}$ & $10$ \\
			$9_{29}$ & \known{$9$} \\
			$9_{30}$ & $10$ \\
			$9_{31}$ & $10$ \\
			$9_{32}$ & $10$ \\
			$9_{33}$ & $10$ \\
			$9_{34}$ & \known{$9$} \\
			$9_{35}$ & \known{$9$} \\
			$9_{36}$ & \newbound{$10$} \\
			$9_{37}$ & $10$ \\
			$9_{38}$ & $10$ \\
			$9_{39}$ & \known{$9$} \\
			$9_{40}$ & \known{$9$} \\
			$9_{41}$ & \known{$9$} \\
			$9_{42}$ & \known{$9$} \\
			$9_{43}$ & \known{$9$} \\
			$9_{44}$ & \known{$9$} \\
			$9_{45}$ & \known{$9$} \\
			$9_{46}$ & \known{$9$} \\
			$9_{47}$ & \known{$9$} \\
			$9_{48}$ & \known{$9$} \\
			$9_{49}$ & \known{$9$} \\[1ex]
		\end{supertabular*}
	\end{multicols*}
\end{center}

\clearpage

\begin{center}

	\begin{multicols*}{5}
		\TrickSupertabularIntoMulticols

		\tablefirsthead{
			$K$ & edges \\
			\midrule
		}
		\tablehead{
			$K$ & edges \\
			\midrule
		}
		\tablelasttail{\bottomrule}

		\setlength{\tabcolsep}{5pt}

		\begin{supertabular*}{.14\textwidth}{lc}
			$10_1$ & $11$ \\ \shrinkheight{-7pt}
			$10_2$ & $11$ \\
			$10_3$ & $11$ \\
			$10_4$ & $11$ \\
			$10_5$ & $11$ \\
			$10_6$ & $11$ \\
			$10_7$ & $11$ \\
			$10_8$ & $10$ \\
			$10_9$ & \newbound{$10$} \\
			$10_{10}$ & \newbound{$10$} \\
			$10_{11}$ & $11$ \\
			$10_{12}$ & $11$ \\
			$10_{13}$ & $11$ \\
			$10_{14}$ & $11$ \\
			$10_{15}$ & $11$ \\
			$10_{16}$ & $10$ \\
			$10_{17}$ & $10$ \\
			$10_{18}$ & $10$ \\
			$10_{19}$ & $10$ \\
			$10_{20}$ & $11$ \\
			$10_{21}$ & $11$ \\
			$10_{22}$ & $11$ \\
			$10_{23}$ & $11$ \\
			$10_{24}$ & $11$ \\
			$10_{25}$ & $11$ \\
			$10_{26}$ & $11$ \\
			$10_{27}$ & $11$ \\
			$10_{28}$ & $11$ \\
			$10_{29}$ & $11$ \\
			$10_{30}$ & $11$ \\
			$10_{31}$ & \newbound{$10$} \\
			$10_{32}$ & $11$ \\
			$10_{33}$ & \newbound{$10$} \\
			$10_{34}$ & $11$ \\
			$10_{35}$ & $11$ \\
			$10_{36}$ & $11$ \\
			$10_{37}$ & \newbound{$11$} \\
			$10_{38}$ & $11$ \\
			$10_{39}$ & $11$ \\
			$10_{40}$ & $11$ \\ \shrinkheight{-7pt}
			$10_{41}$ & $11$ \\
			$10_{42}$ & $11$ \\
			$10_{43}$ & $11$ \\
			$10_{44}$ & $11$ \\
			$10_{45}$ & $11$ \\
			$10_{46}$ & $11$ \\
			$10_{47}$ & $11$ \\
			$10_{48}$ & $10$ \\
			$10_{49}$ & $11$ \\
			$10_{50}$ & $11$ \\
			$10_{51}$ & $11$ \\
			$10_{52}$ & \newbound{$10$} \\
			$10_{53}$ & $11$ \\
			$10_{54}$ & $11$ \\
			$10_{55}$ & $11$ \\
			$10_{56}$ & $10$ \\
			$10_{57}$ & $11$ \\
			$10_{58}$ & $11$ \\
			$10_{59}$ & $11$ \\
			$10_{60}$ & \newbound{$10$} \\
			$10_{61}$ & \newbound{$10$} \\
			$10_{62}$ & $11$ \\
			$10_{63}$ & $11$ \\
			$10_{64}$ & $11$ \\
			$10_{65}$ & $11$ \\
			$10_{66}$ & $11$ \\
			$10_{67}$ & $11$ \\
			$10_{68}$ & $10$ \\
			$10_{69}$ & $11$ \\
			$10_{70}$ & $11$ \\
			$10_{71}$ & $11$ \\
			$10_{72}$ & $11$ \\
			$10_{73}$ & $11$ \\
			$10_{74}$ & $11$ \\ \shrinkheight{-7pt}
			$10_{75}$ & $11$ \\
			$10_{76}$ & \newbound{$11$} \\
			$10_{77}$ & $11$ \\
			$10_{78}$ & $11$ \\
			$10_{79}$ & $11$ \\
			$10_{80}$ & $11$ \\
			$10_{81}$ & $11$ \\
			$10_{82}$ & $10$ \\
			$10_{83}$ & $10$ \\
			$10_{84}$ & $10$ \\
			$10_{85}$ & $10$ \\
			$10_{86}$ & \newbound{$10$} \\
			$10_{87}$ & \newbound{$10$} \\
			$10_{88}$ & $11$ \\
			$10_{89}$ & $11$ \\
			$10_{90}$ & $10$ \\
			$10_{91}$ & $10$ \\
			$10_{92}$ & \newbound{$10$} \\
			$10_{93}$ & $10$ \\
			$10_{94}$ & $10$ \\
			$10_{95}$ & $11$ \\
			$10_{96}$ & \newbound{$10$} \\
			$10_{97}$ & \newbound{$10$} \\
			$10_{98}$ & $11$ \\
			$10_{99}$ & $11$ \\
			$10_{100}$ & $10$ \\
			$10_{101}$ & \newbound{$10$} \\
			$10_{102}$ & $10$ \\
			$10_{103}$ & $10$ \\
			$10_{104}$ & $10$ \\
			$10_{105}$ & $10$ \\
			$10_{106}$ & $10$ \\
			$10_{107}$ & $10$ \\
			$10_{108}$ & $10$ \\
			$10_{109}$ & $10$ \\
			$10_{110}$ & $10$ \\
			$10_{111}$ & $10$ \\
			$10_{112}$ & $10$ \\
			$10_{113}$ & $10$ \\ \shrinkheight{-7pt}
			$10_{114}$ & $10$ \\
			$10_{115}$ & $10$ \\
			$10_{116}$ & $10$ \\
			$10_{117}$ & $10$ \\
			$10_{118}$ & $10$ \\
			$10_{119}$ & $10$ \\
			$10_{120}$ & $10$ \\
			$10_{121}$ & $10$ \\
			$10_{122}$ & $10$ \\
			$10_{123}$ & \newbound{$10$} \\
			$10_{124}$ & \known{$10$} \\
			$10_{125}$ & $10$ \\
			$10_{126}$ & $10$ \\
			$10_{127}$ & $10$ \\
			$10_{128}$ & $10$ \\
			$10_{129}$ & $10$ \\
			$10_{130}$ & $10$ \\
			$10_{131}$ & $10$ \\
			$10_{132}$ & $10$ \\
			$10_{133}$ & $10$ \\
			$10_{134}$ & $10$ \\
			$10_{135}$ & $10$ \\
			$10_{136}$ & $10$ \\
			$10_{137}$ & $10$ \\
			$10_{138}$ & $10$ \\
			$10_{139}$ & $10$ \\
			$10_{140}$ & $10$ \\
			$10_{141}$ & $10$ \\
			$10_{142}$ & $10$ \\
			$10_{143}$ & $10$ \\
			$10_{144}$ & $10$ \\
			$10_{145}$ & $10$ \\
			$10_{146}$ & $10$ \\
			$10_{147}$ & $10$ \\
			$10_{148}$ & $10$ \\
			$10_{149}$ & $10$ \\
			$10_{150}$ & $10$ \\
			$10_{151}$ & $10$ \\
			$10_{152}$ & $10$ \\
			$10_{153}$ & $10$ \\
			$10_{154}$ & \newbound{$10$} \\
			$10_{155}$ & $10$ \\
			$10_{156}$ & $10$ \\
			$10_{157}$ & $10$ \\
			$10_{158}$ & $10$ \\
			$10_{159}$ & $10$ \\
			$10_{160}$ & $10$ \\
			$10_{161}$ & $10$ \\
			$10_{162}$ & $10$ \\
			$10_{163}$ & $10$ \\
			$10_{164}$ & $10$ \\
			$10_{165}$ & $10$ \\[1ex]
		\end{supertabular*}
	\end{multicols*}
\end{center}
%\clearpage

\section{From Diagrams to (Nearly) Minimal Stick Realizations of Large Random Knots}
\label{sec:large knot flow}

\begin{figure}[ht]
	\label{fig:complicated examples 1}
	\begin{minipage}{0.45\textwidth}
		\centering
		\includegraphics[width=\textwidth]{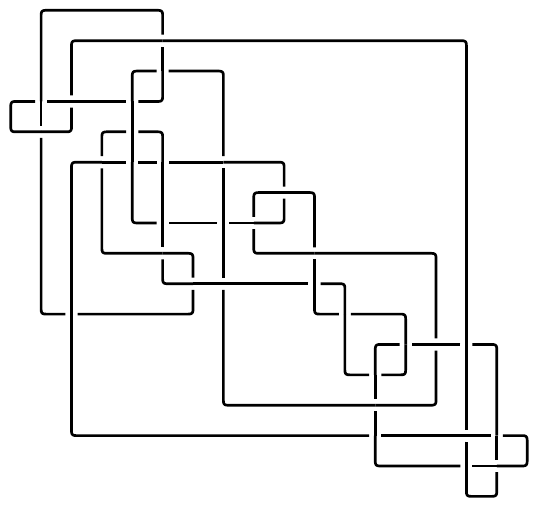}\\
		$28 \leq c[K] \leq 32$\\
		$\stick[K] \leq 25$
	\end{minipage}
	\hfil
	\begin{minipage}{0.45\textwidth}
		\centering
		\includegraphics[width=\textwidth]{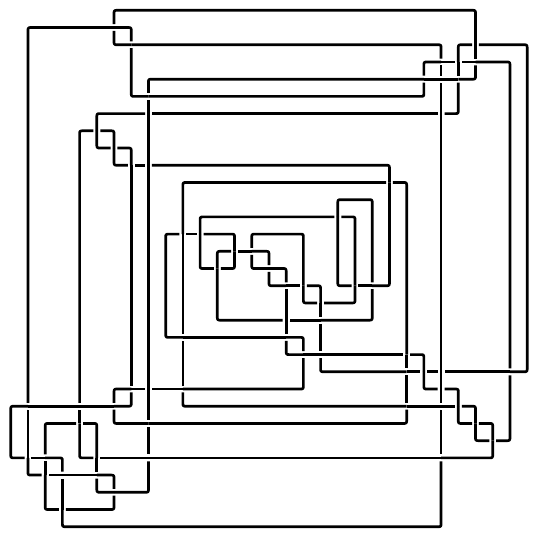}\\
		$46 \leq c[K] \leq 64$\\
		$\stick[K] \leq 40$
	\end{minipage} \\[2ex]
	\begin{minipage}{0.45\textwidth}
		\centering
		\includegraphics[width=\textwidth]{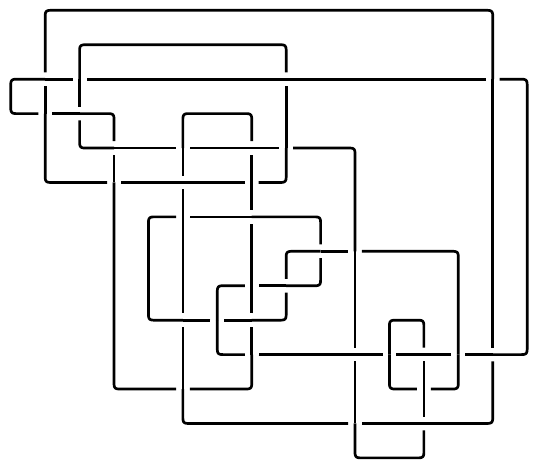}\\
		$c[K] = 32$ \\
		$\stick[K] \leq 23$
	\end{minipage}
	\hfil
	\begin{minipage}{0.45\textwidth}
		\centering
		\includegraphics[width=\textwidth]{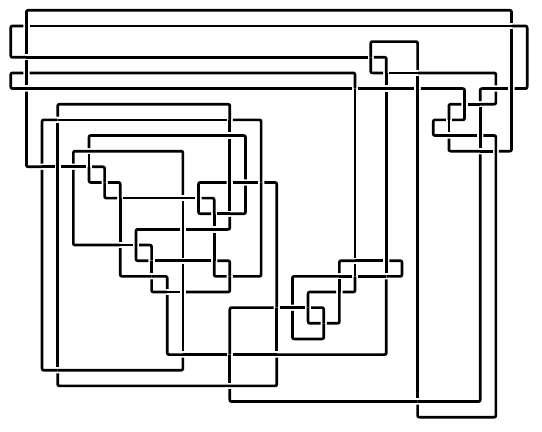}\\
		$c[K] = 64$ \\
		$\stick[K] \leq 48$
	\end{minipage}
	\caption{Initial diagrams of four large, untabulated knots. The knots in the top row are non-alternating, and those in the bottom row are alternating.}
\end{figure}

\begin{figure}
	\begin{center}
		\begin{minipage}{0.45\textwidth}
			\centering
			\includegraphics[width=\textwidth]{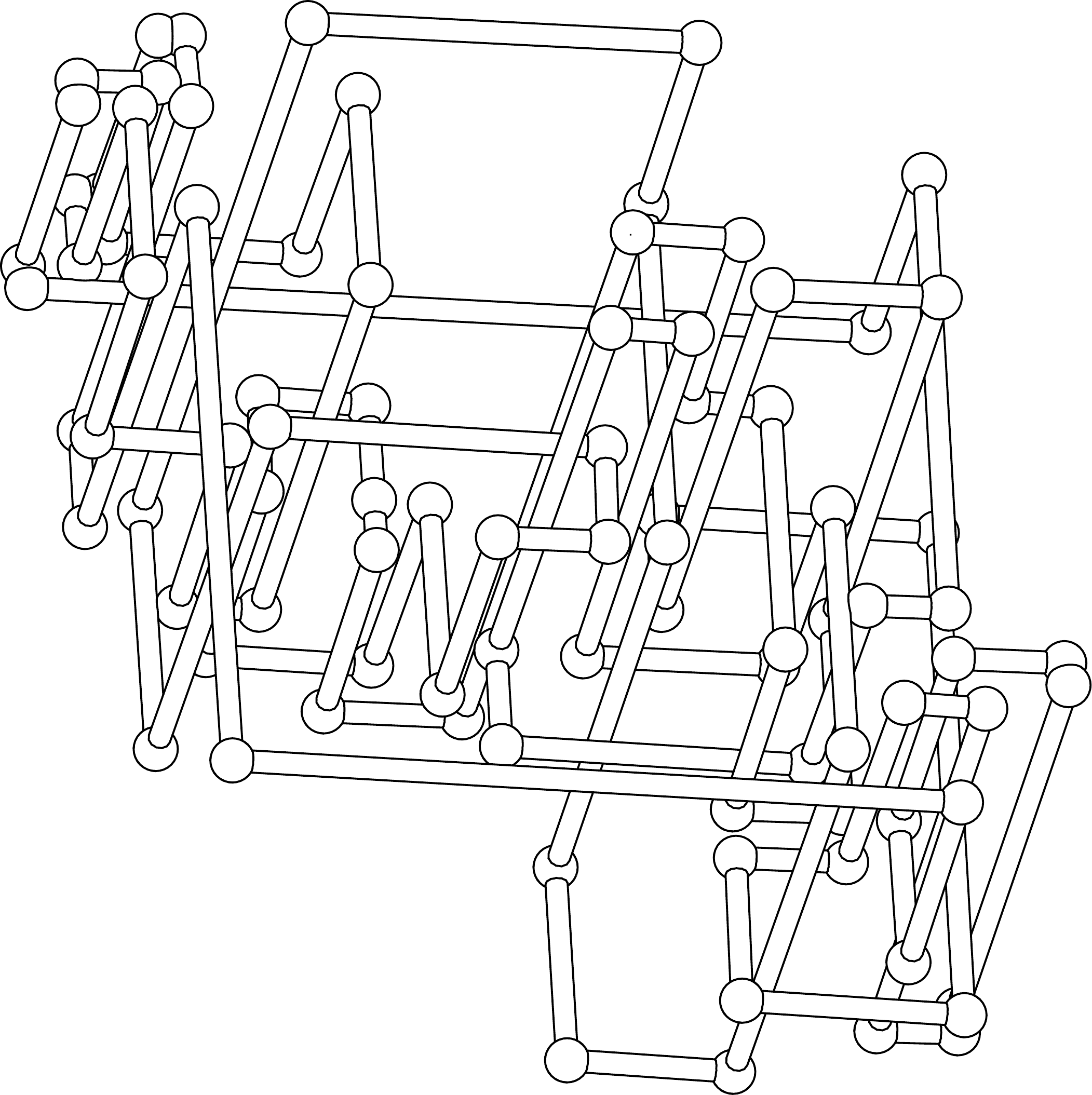} \\
			9242 lattice edges, 90 corners
		\end{minipage}
		\hfil
		\begin{minipage}{0.45\textwidth}
			\centering
			\includegraphics[width=\textwidth]{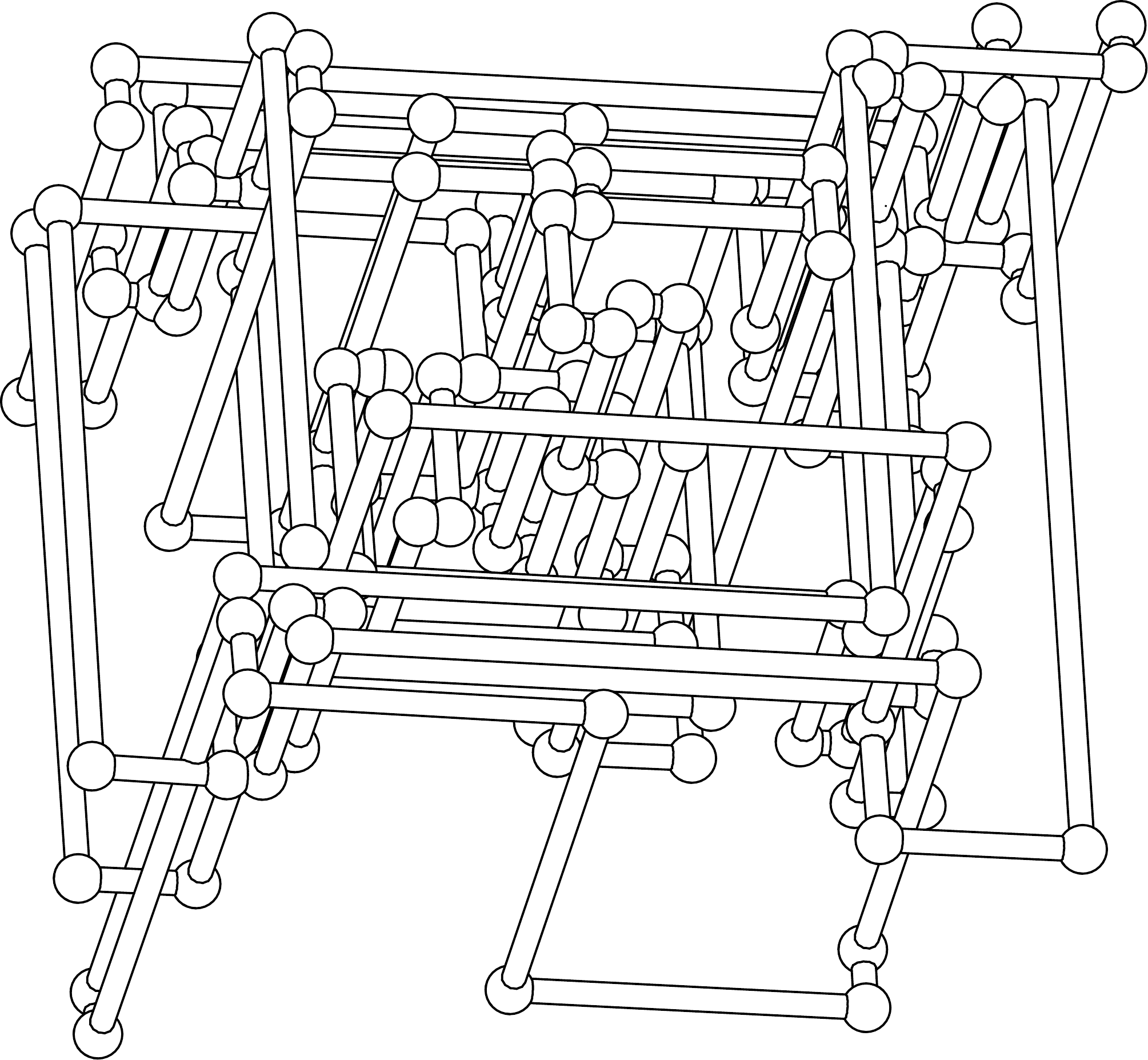} \\
			25720 lattice edges, 156 corners
		\end{minipage} \\[2ex]
		\begin{minipage}{0.45\textwidth}
			\centering
			\includegraphics[width=\textwidth]{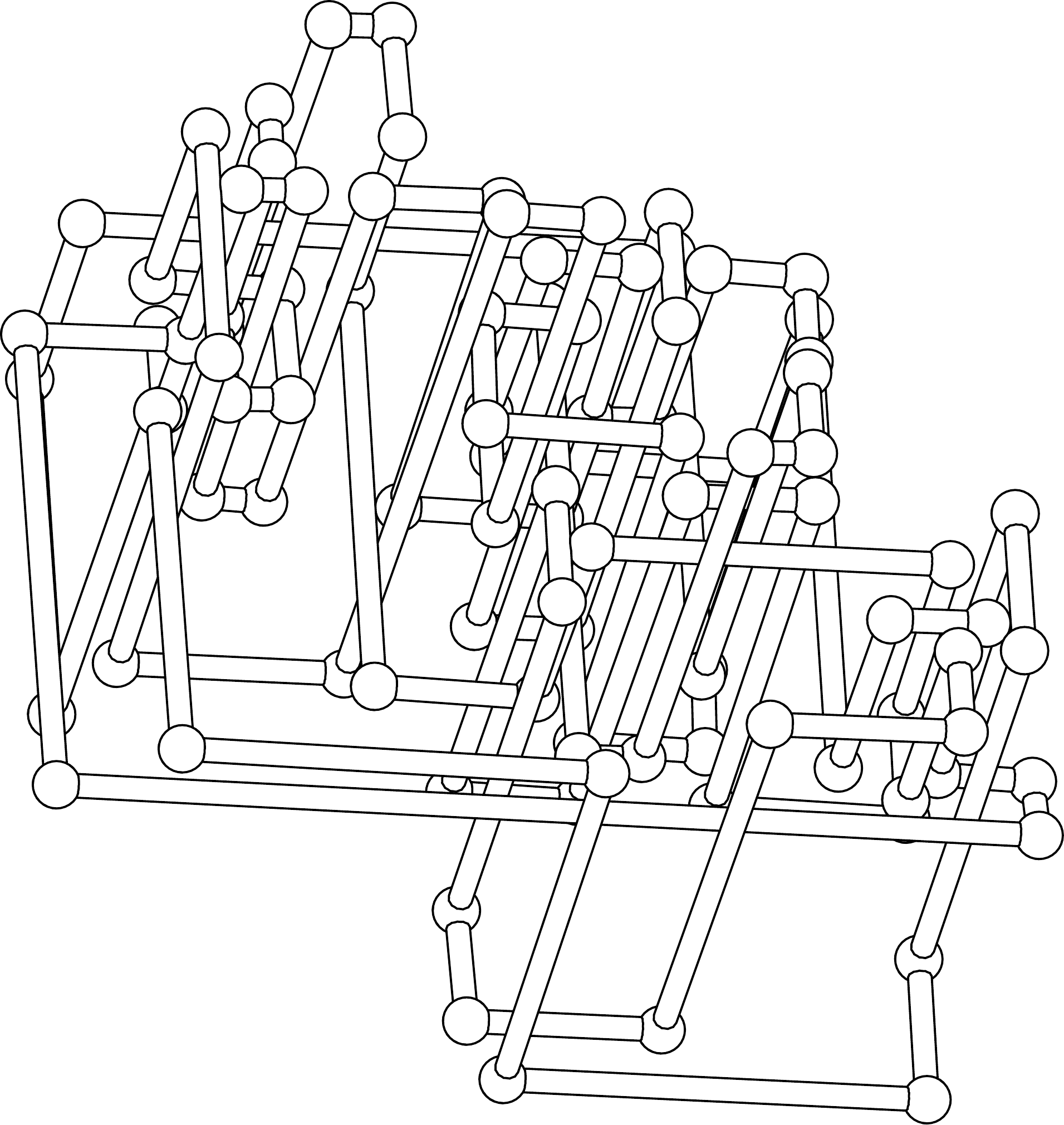} \\
			9212 lattice edges, 100 corners
		\end{minipage}
		\hfil
		\begin{minipage}{0.45\textwidth}
			\centering
			\includegraphics[width=\textwidth]{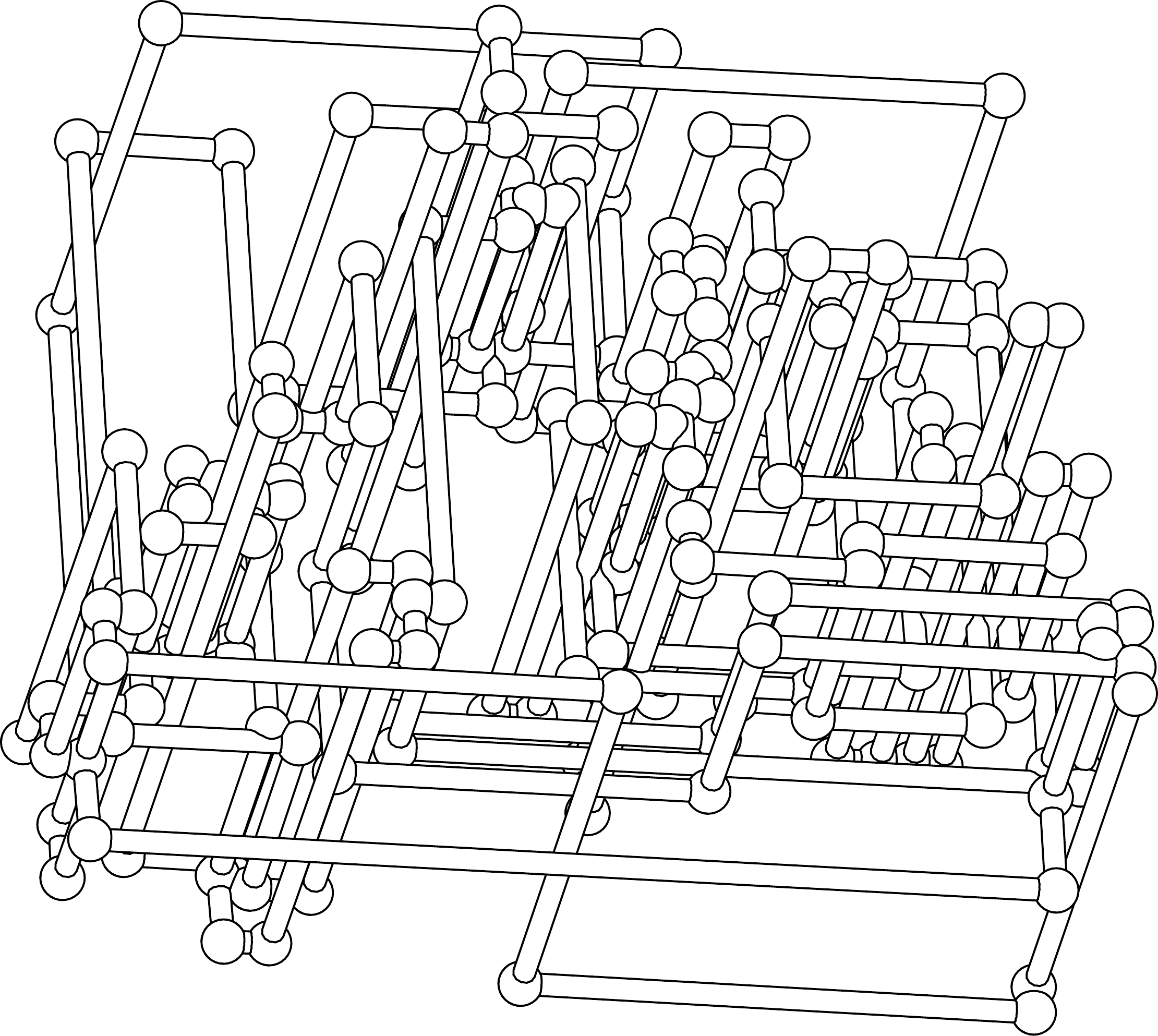} \\
			34420 lattice edges, 208 corners
		\end{minipage}
	\end{center}
	\caption{Initial lattice embeddings of the four untabulated knots given in Figure~\ref{fig:complicated examples 1}.
		These are far from minimal, containing between 9k and 34k lattice edges; accordingly we do not show all vertices but only corners.
	}\label{fig:complicated examples 2}
\end{figure}

\begin{figure}
	\begin{center}
		\begin{minipage}{0.45\textwidth}
			\centering
			\includegraphics[width=\textwidth]{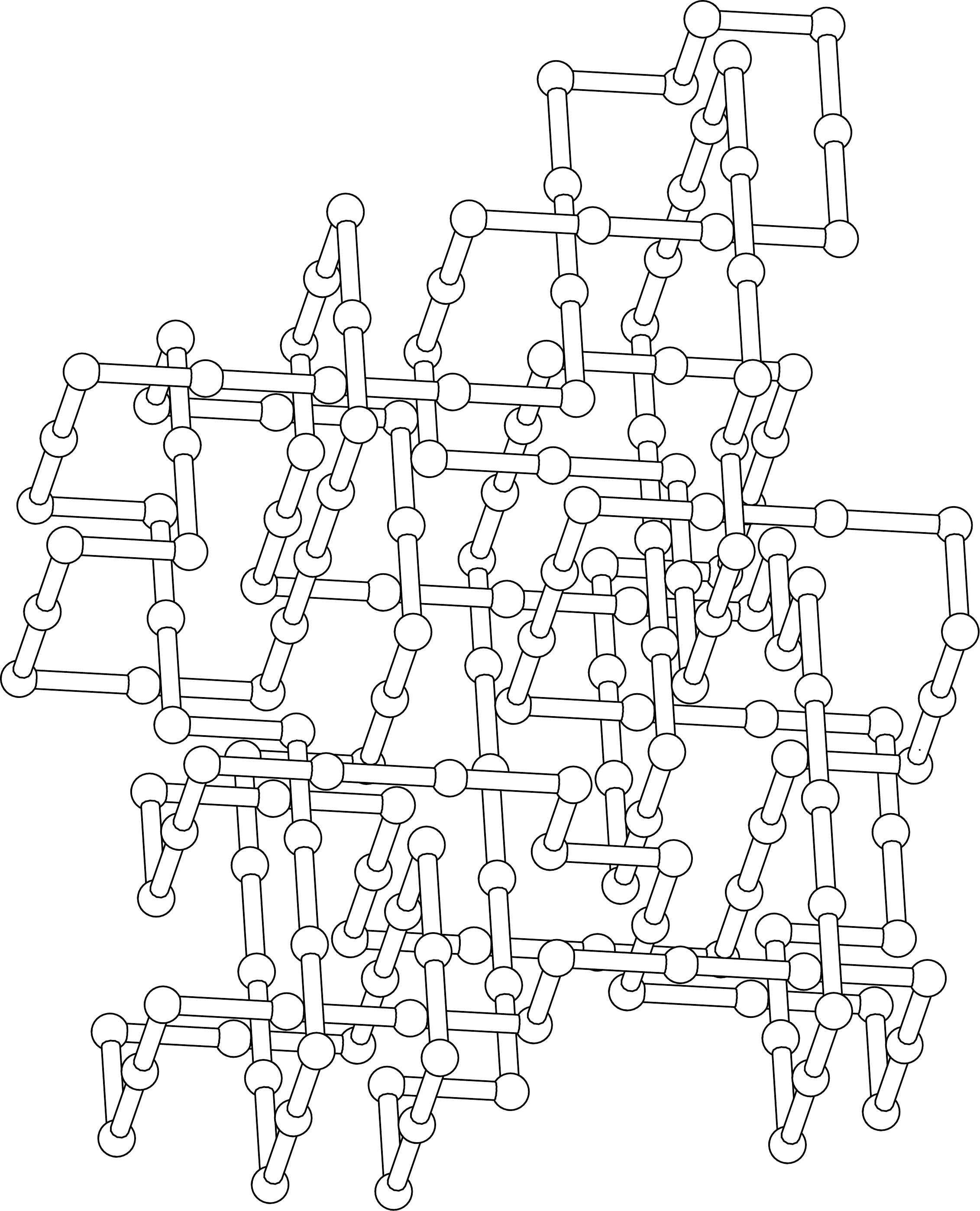} \\
			174 lattice edges
		\end{minipage}
		\hfil
		\begin{minipage}{0.45\textwidth}
			\centering
			\includegraphics[width=\textwidth]{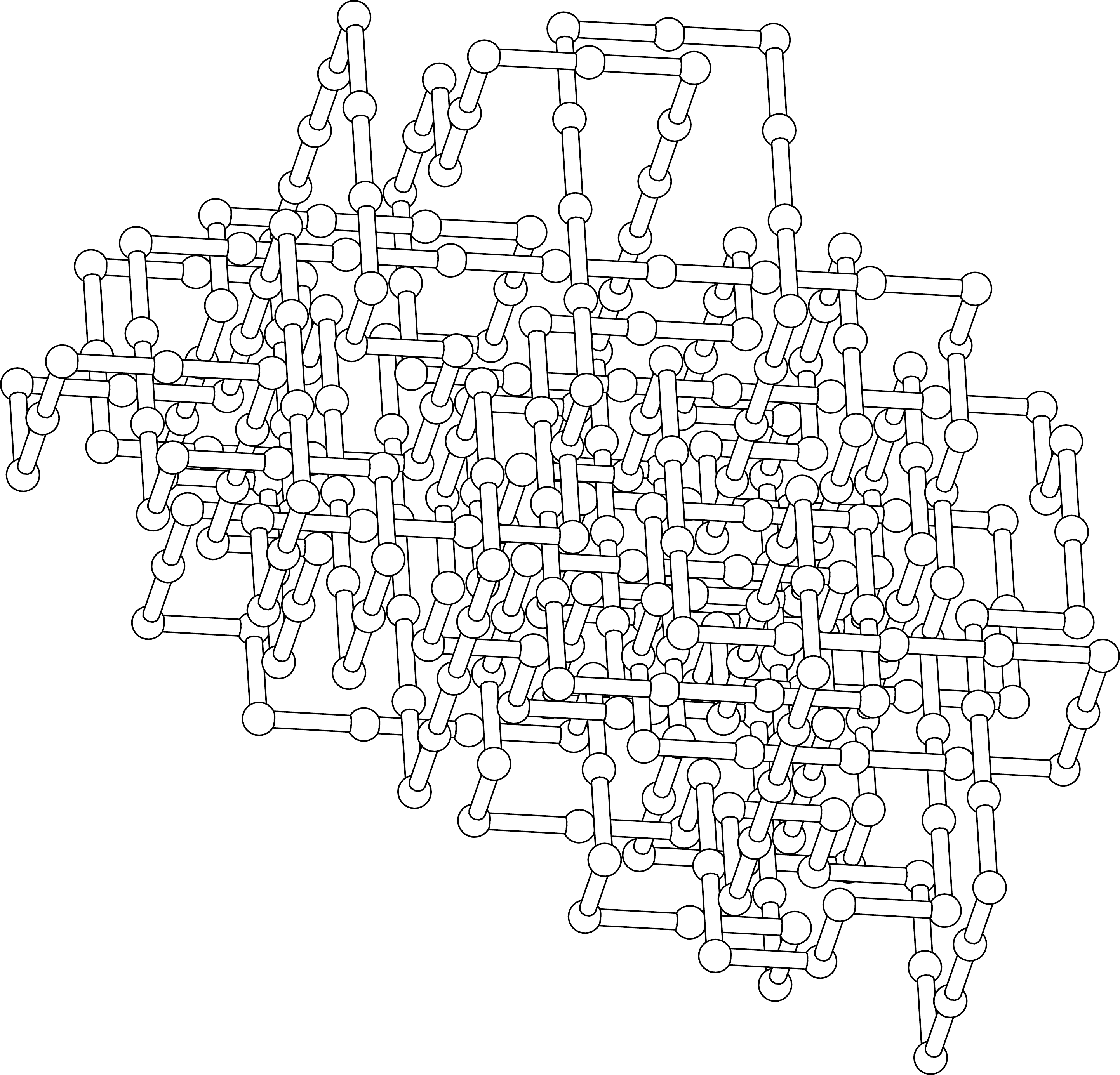} \\
			310 lattice edges
		\end{minipage} \\[2ex]
		\begin{minipage}{0.45\textwidth}
			\centering
			\includegraphics[width=\textwidth]{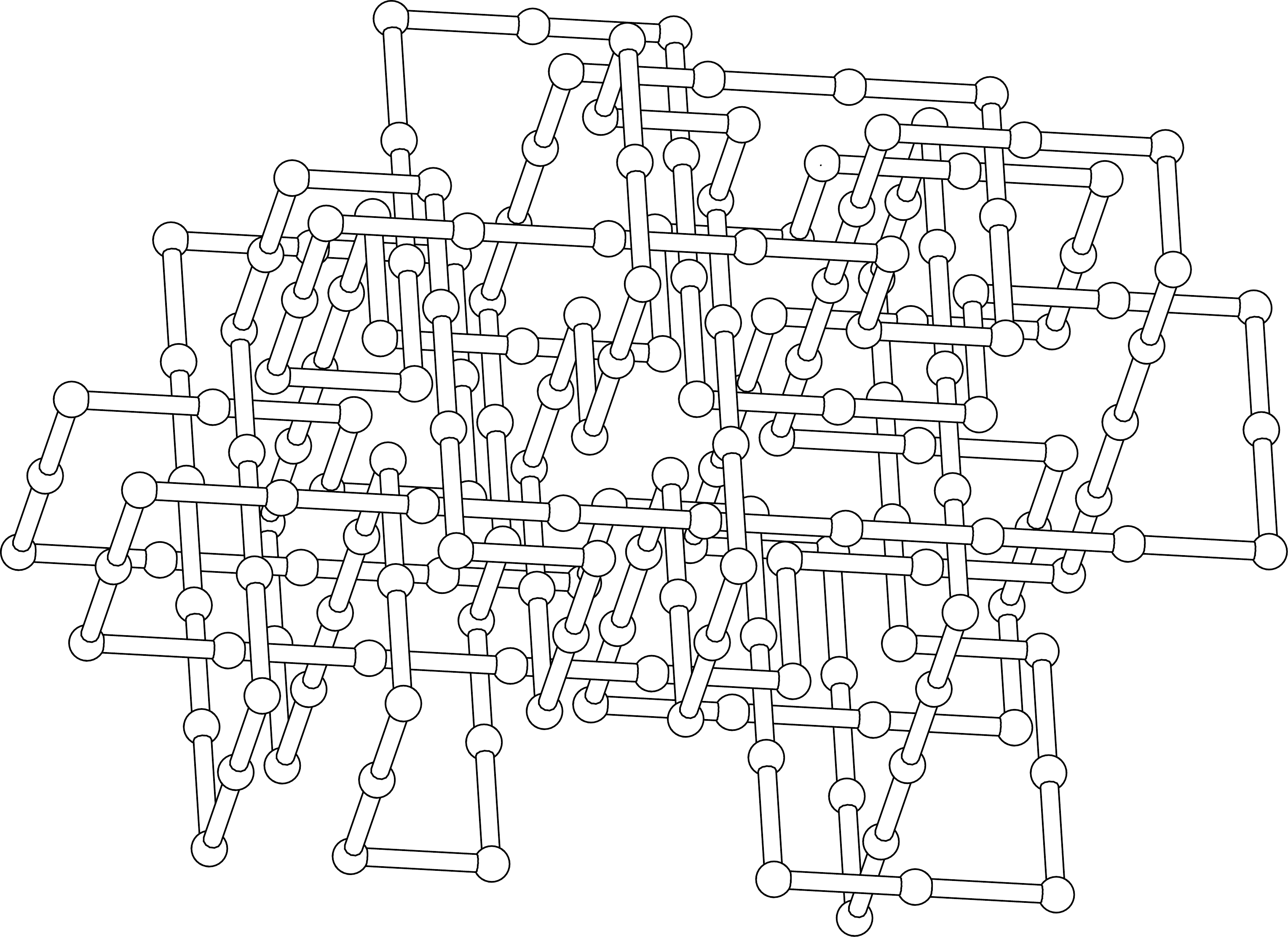} \\
			186 lattice edges
		\end{minipage}
		\hfil
		\begin{minipage}{0.45\textwidth}
			\centering
			\includegraphics[width=\textwidth]{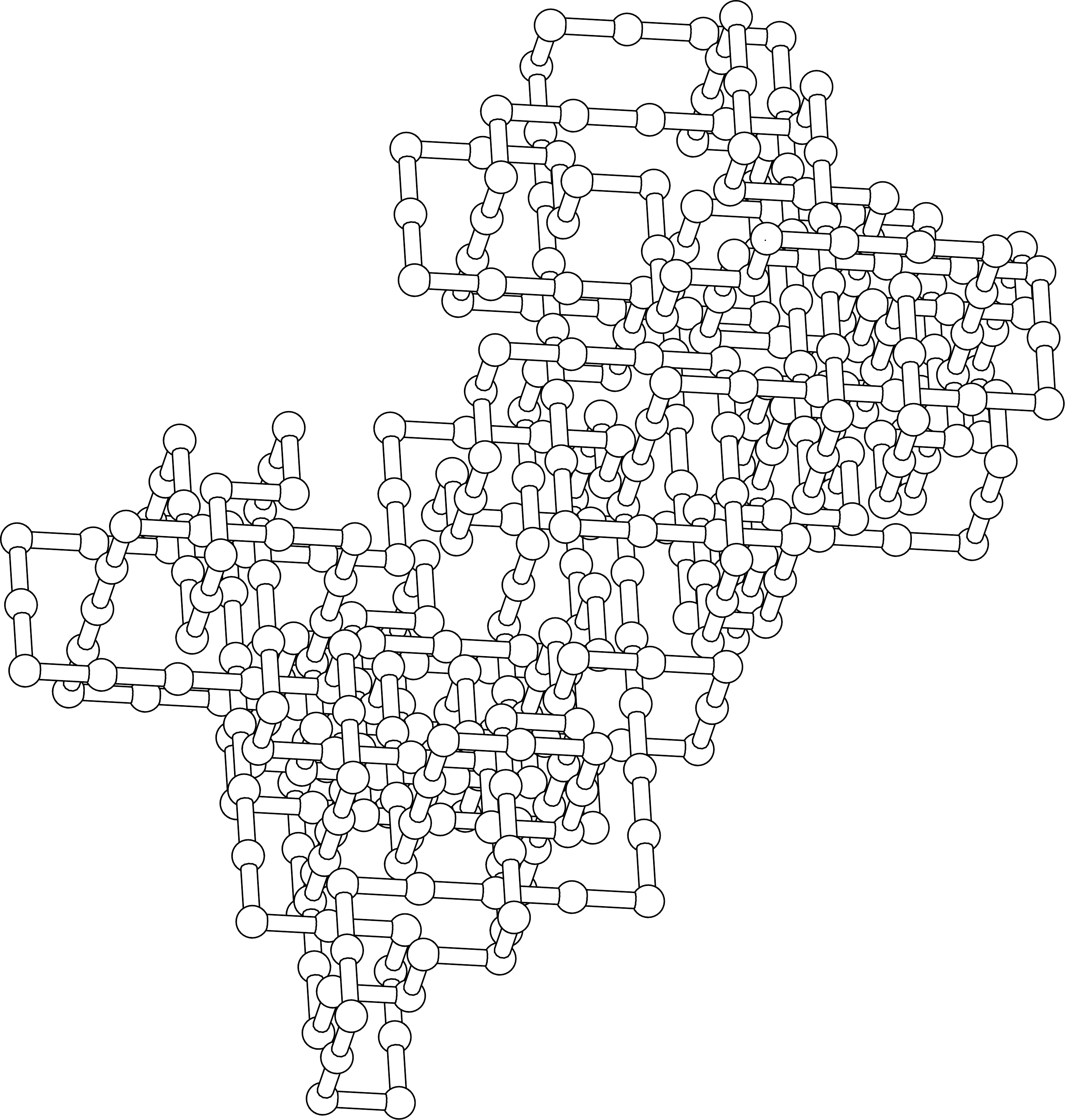} \\
			382 lattice edges
		\end{minipage}
	\end{center}
	\caption{Lattice embeddings of the four untabulated knots given in Figure~\ref{fig:complicated examples 1}.
		These are not the initial lattice embeddings, but have been reduced to a few hundred vertices using the
		BFACF algorithm as described in Section~\ref{sec:methods}.
	}\label{fig:complicated examples 3}
\end{figure}

\begin{figure}
	\begin{center}
		\begin{minipage}{0.45\textwidth}
			\centering
			\includegraphics[width=\textwidth]{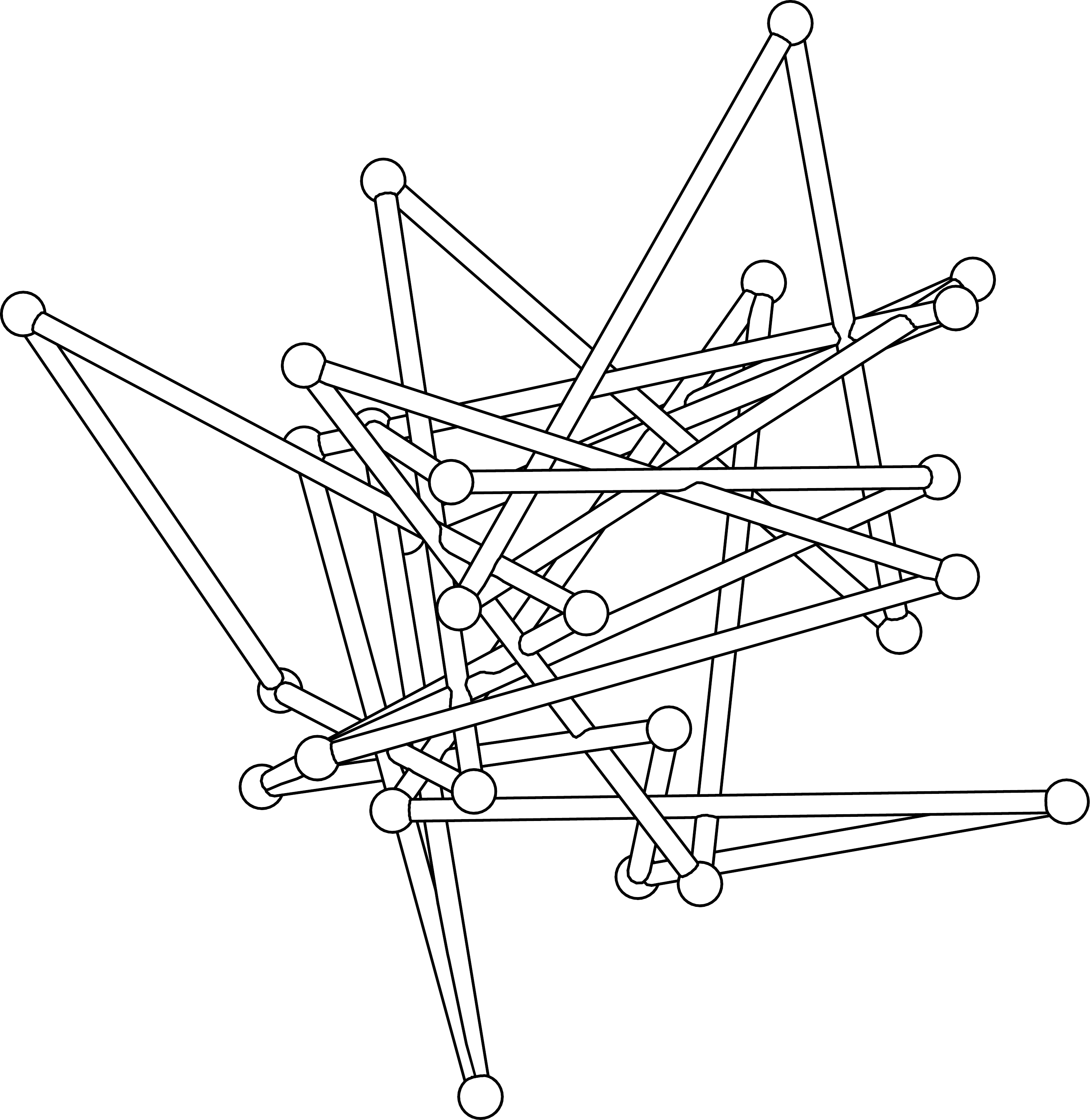}\\
			26 unit sticks
		\end{minipage}
		\hfil
		\begin{minipage}{0.45\textwidth}
			\centering
			\includegraphics[width=\textwidth]{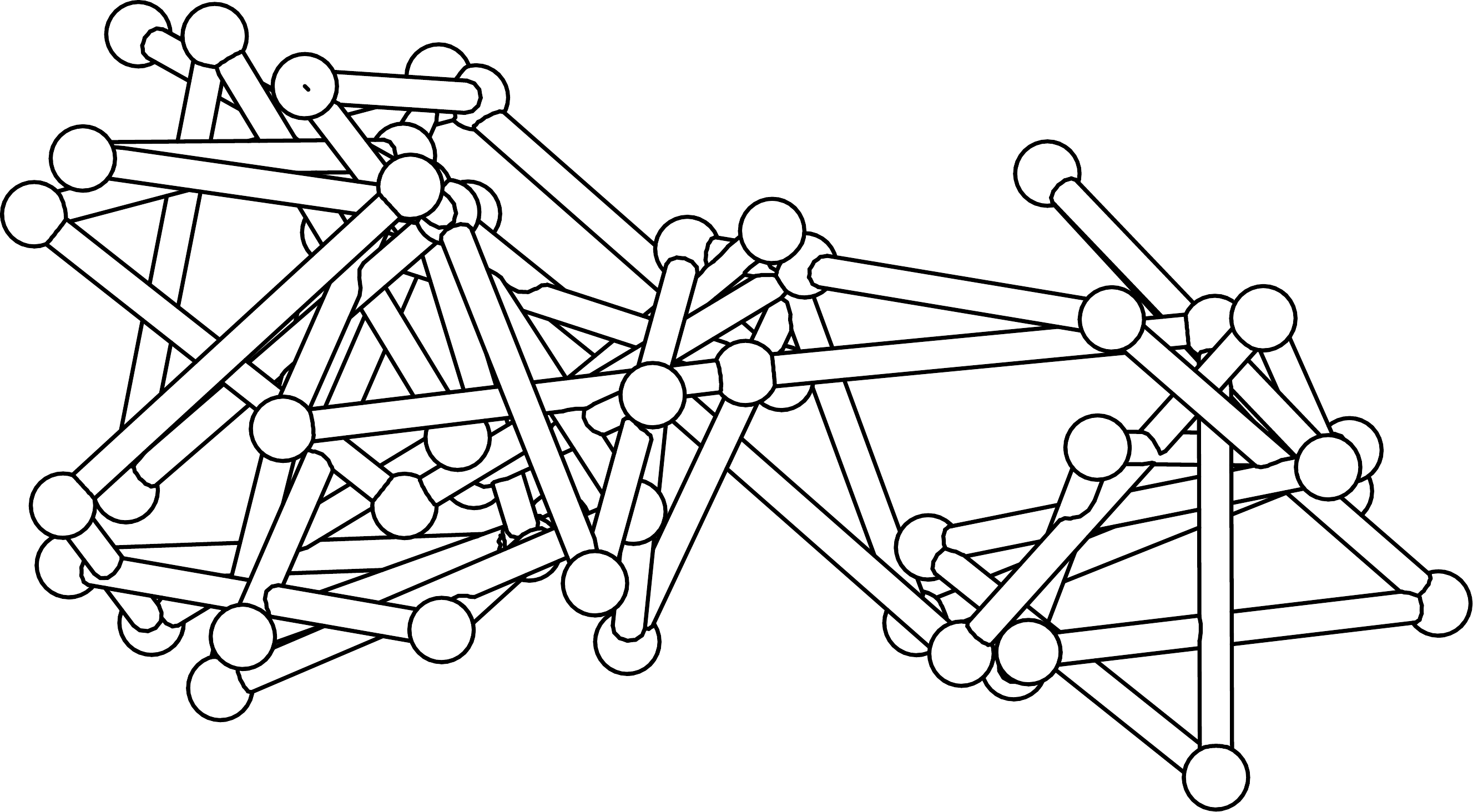}\\
			52 unit sticks
		\end{minipage} \\[2ex]
		\begin{minipage}{0.45\textwidth}
			\centering
			\includegraphics[width=\textwidth]{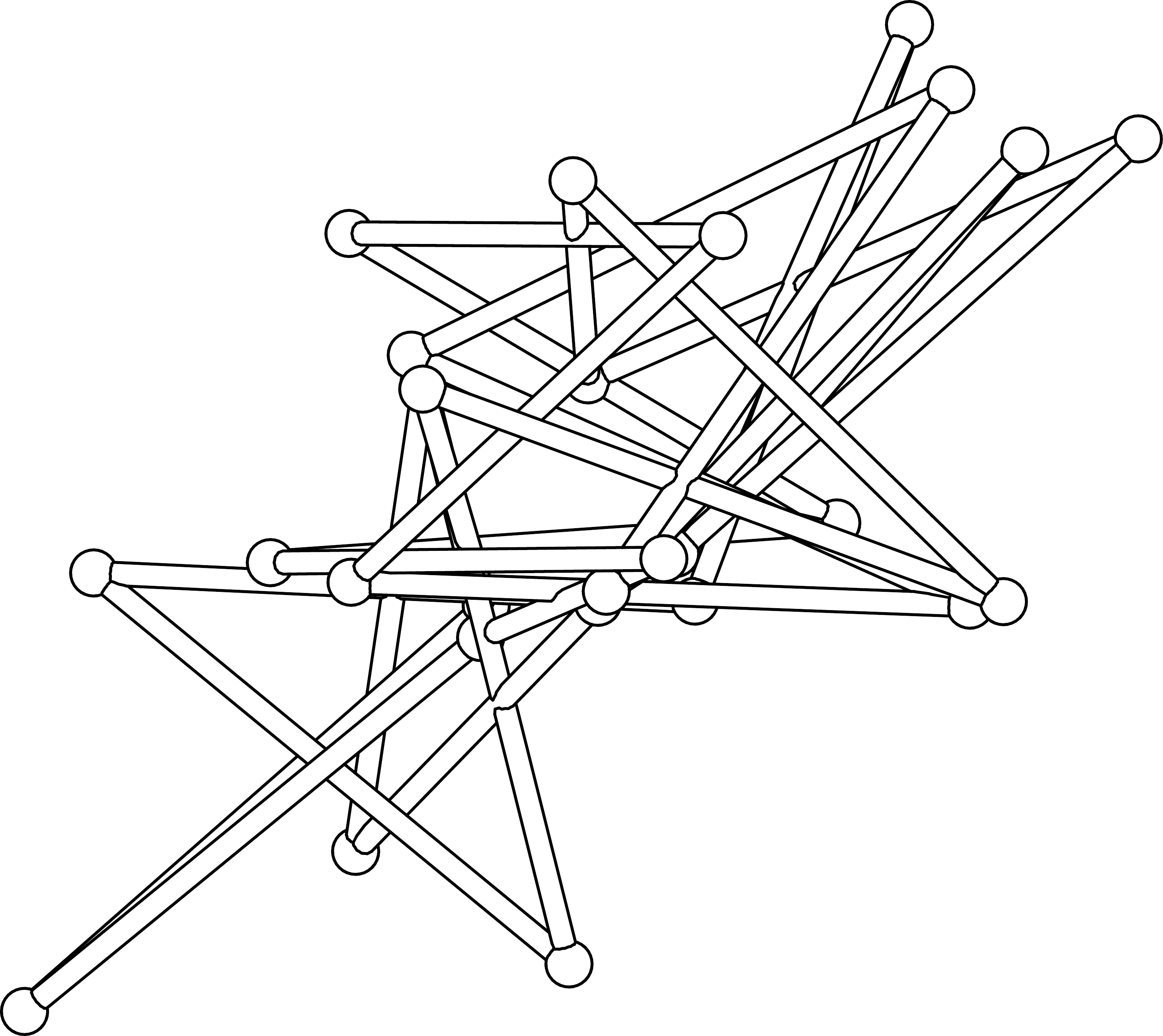} \\
			28 unit sticks
		\end{minipage}
		\hfil
		\begin{minipage}{0.45\textwidth}
			\centering
			\includegraphics[width=\textwidth]{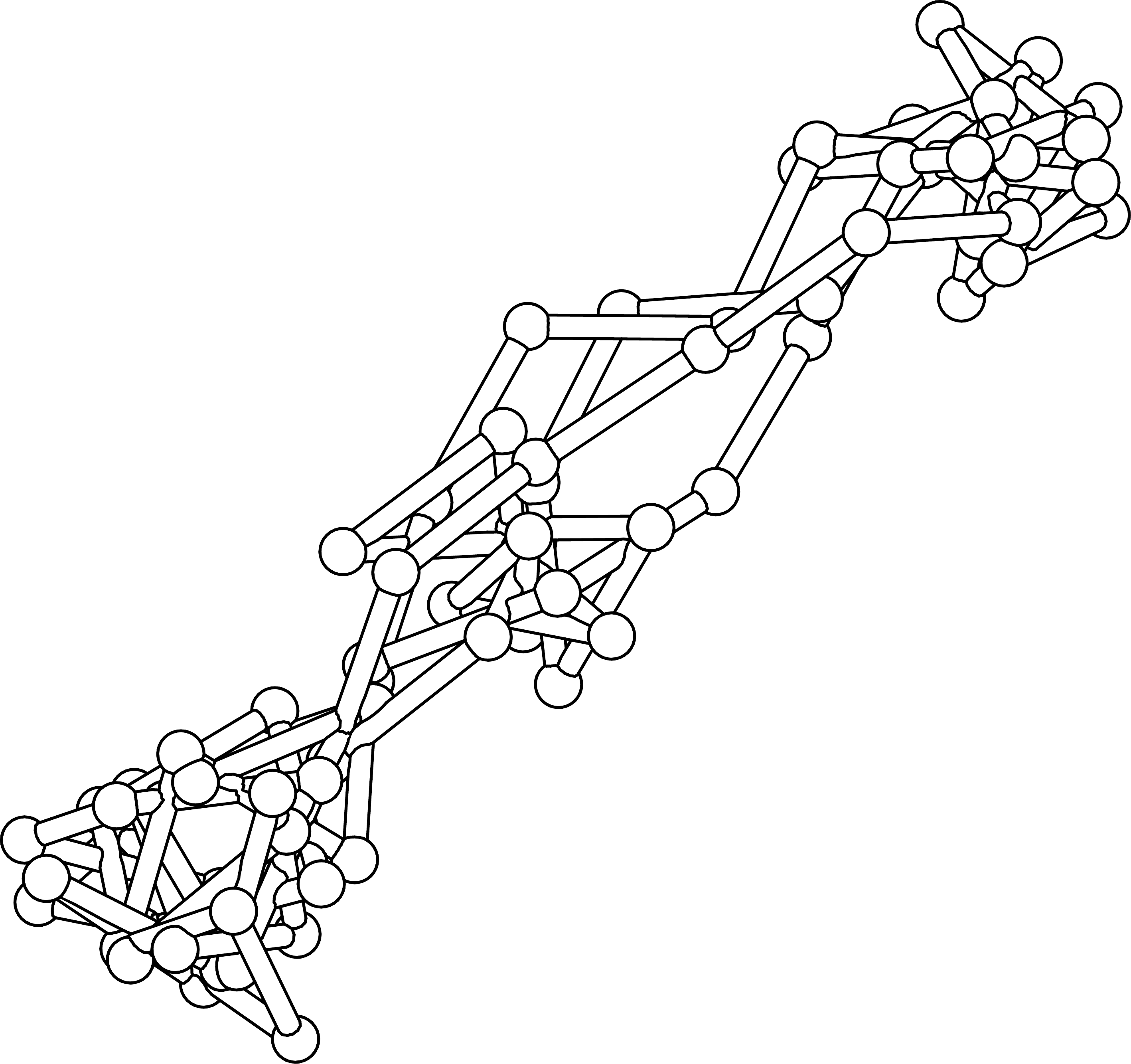} \\
			74 unit sticks
		\end{minipage}
	\end{center}
	\caption{Unit stick embeddings of the four untabulated knots from Figure~\ref{fig:complicated examples 1}. We
		do not claim that these are minimal, but we think those of the smaller knots are not too far from minimal,
		while the larger knots appear to need more minimisation.
	}\label{fig:complicated examples 4}
\end{figure}

\begin{figure}
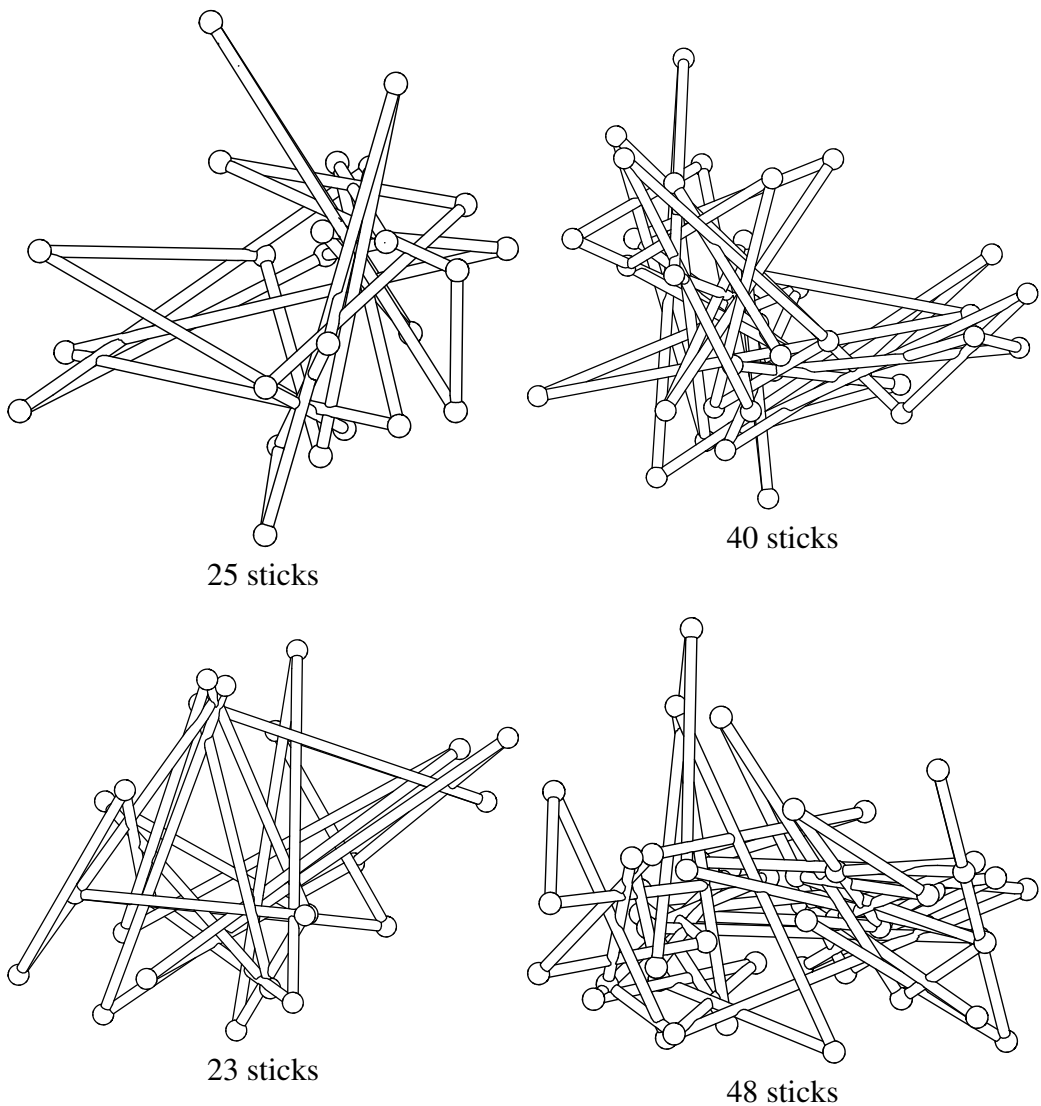

	\begin{center}
		\begin{minipage}{0.45\textwidth}
			\centering
			\includegraphics[width=\textwidth]{figs/stick_n32.pdf} \\
			25 sticks
		\end{minipage}
		\begin{minipage}{0.45\textwidth}
			\centering
			\includegraphics[width=\textwidth]{figs/stick_n64.pdf} \\
			40 sticks
		\end{minipage} \\[2ex]
		\begin{minipage}{0.45\textwidth}
			\centering
			\includegraphics[width=\textwidth]{figs/stick_a32.pdf} \\
			23 sticks
		\end{minipage}
		\begin{minipage}{0.45\textwidth}
			\centering
			\includegraphics[width=\textwidth]{figs/stick_a64.pdf} \\
			48 sticks
		\end{minipage}
	\end{center}
	\caption{Stick embeddings of the four untabulated knots from \autoref{fig:complicated examples 1}.
		We believe that should not be too far from the minimal embeddings. We also note that in each case the bound on stick number is smaller than
		the crossing number.
	}\label{fig:complicated examples 5}
\end{figure}

\clearpage

\bibliography{stickknots-special,stickknots}

\begin{thebibliography}{10}

\bibitem{adamsStickNumbersComposition1997}
Colin Adams, Bevin~M. Brennan, Deborah~L. Greilsheimer, and Alexander~K. Woo.
\newblock Stick numbers and composition of knots and links.
\newblock {\em Journal of Knot Theory and Its Ramifications}, 6(2):149--161,
  1997.

\bibitem{aragaodecarvalhoPolymers$g|varphi|^4$Theory1983}
Carlos {Arag{\~a}o de Carvalho}, Sergio Caracciolo, and J{\"o}rg Fr{\"o}hlich.
\newblock Polymers and {$g|\varphi|^4$} theory in four dimensions.
\newblock {\em Nuclear Physics B}, 215(2):209--248, 1983.

\bibitem{aragaodecarvalhoNewMonteCarloApproach1983}
Carlos Arag{\~a}o De~Carvalho and Sergio Caracciolo.
\newblock A new {{Monte-Carlo}} approach to the critical properties of
  self-avoiding random walks.
\newblock {\em Journal de Physique}, 44(3):323--331, 1983.

\bibitem{ashtonKnotTighteningConstrained2011}
Ted Ashton, Jason Cantarella, Michael Piatek, and Eric~J. Rawdon.
\newblock Knot tightening by constrained gradient descent.
\newblock {\em Experimental Mathematics}, 20(1):57--90, 2011.

\bibitem{baaderCoxeterGroupsMeridional2021}
Sebastian Baader, Ryan Blair, and Alexandra Kjuchukova.
\newblock Coxeter groups and meridional rank of links.
\newblock {\em Mathematische Annalen}, 379(3-4):1533--1551, April 2021.

\bibitem{bergRandomPathsRandom1981}
Bernd~A. Berg and D.~Foerster.
\newblock Random paths and random surfaces on a digital computer.
\newblock {\em Physics Letters B}, 106(4):323--326, November 1981.

\bibitem{blairKnotsExactly102020}
Ryan Blair, Thomas~D. Eddy, Nathaniel Morrison, and Clayton Shonkwiler.
\newblock Knots with exactly 10 sticks.
\newblock {\em Journal of Knot Theory and Its Ramifications}, 29(3):2050011,
  2020.

\bibitem{blairCoxeterQuotientsKnot2025}
Ryan Blair, Alexandra Kjuchukova, and Nathaniel Morrison.
\newblock Coxeter quotients of knot groups through 16 crossings.
\newblock {\em Experimental Mathematics}, 34(1):27--37, 2025.

\bibitem{boileauNombrePontsGenerateurs1985}
Michel Boileau and Heiner Zieschang.
\newblock Nombre de ponts et g{\'e}n{\'e}rateurs m{\'e}ridiens des entrelacs de
  {{Montesinos}}.
\newblock {\em Commentarii Mathematici Helvetici}, 60(2):270--279, 1985.

\bibitem{CGAL}
Herv{\'e} Br{\"o}nnimann, Andreas Fabri, Geert-Jan Giezeman, Susan Hert,
  Michael Hoffmann, Lutz Kettner, Sylvain Pion, and Stefan Schirra.
\newblock {2D} and {3D} linear geometry kernel.
\newblock In {\em {CGAL} User and Reference Manual}. {CGAL Editorial Board},
  {6.0.1} edition, 2024.

\bibitem{buckThicknessCrossingNumber1999}
Gregory Buck and Jonathan~K. Simon.
\newblock Thickness and crossing number of knots.
\newblock {\em Topology and its Applications}, 91(3):245--257, 1999.

\bibitem{burtonNext350Million2020}
Benjamin~A. Burton.
\newblock The next 350 million knots.
\newblock {\em LIPIcs, Volume 164, SoCG 2020}, 164:25:1--25:17, 2020.

\bibitem{calvoGeometricKnotTheory1998}
Jorge~Alberto Calvo.
\newblock {\em Geometric {{Knot Theory}}: {{The Classification}} of {{Spatial
  Polygons}} with a {{Small Number}} of {{Edges}}}.
\newblock PhD thesis, University of California, Santa Barbara, 1998.

\bibitem{calvoGeometricKnotSpaces2001}
Jorge~Alberto Calvo.
\newblock Geometric knot spaces and polygonal isotopy.
\newblock {\em Journal of Knot Theory and Its Ramifications}, 10(2):245--267,
  2001.

\bibitem{calvoCharacterizingPolygonsBbb2002}
Jorge~Alberto Calvo.
\newblock Characterizing polygons in {$\Bbb R^3$}.
\newblock In Jorge~Alberto Calvo, Kenneth~C. Millett, and Eric~J. Rawdon,
  editors, {\em Physical {{Knots}}: {{Knotting}}, {{Linking}}, and {{Folding
  Geometric Objects}} in $\Bbb R^3$}, volume 304 of {\em Contemporary
  {{Mathematics}}}, pages 37--53. American Mathematical Society, Providence,
  RI, USA, 2002.

\bibitem{calvoMinimalEdgePiecewise1998}
Jorge~Alberto Calvo and Kenneth~C. Millett.
\newblock Minimal edge piecewise linear knots.
\newblock In Andrzej Stasiak, Vsevolod Katritch, and Louis~H. Kauffman,
  editors, {\em Ideal {{Knots}}}, volume~19 of {\em Series on {{Knots}} and
  {{Everything}}}, pages 107--128. World Scientific Publishing, Singapore,
  1998.

\bibitem{dataverse-pretzel}
Jason Cantarella, Andrew Rechnitzer, Henrik Schumacher, and Clayton Shonkwiler.
\newblock {Low stick number polygons for a family of pretzel knots}.
\newblock \url{https://doi.org/10.7910/DVN/S0JQVH}, 2025.

\bibitem{dataverse-big}
Jason Cantarella, Andrew Rechnitzer, Henrik Schumacher, and Clayton Shonkwiler.
\newblock {Low stick number polygons for four large knots}.
\newblock \url{https://doi.org/10.7910/DVN/JXGUXD}, 2025.

\bibitem{dataverse-torus}
Jason Cantarella, Andrew Rechnitzer, Henrik Schumacher, and Clayton Shonkwiler.
\newblock {Low stick number polygons for torus knots}.
\newblock \url{https://doi.org/10.7910/DVN/GCNJLI}, 2025.

\bibitem{dataverse-twist}
Jason Cantarella, Andrew Rechnitzer, Henrik Schumacher, and Clayton Shonkwiler.
\newblock {Low stick number polygons for twist knots}.
\newblock \url{https://doi.org/10.7910/DVN/QMQ9R0}, 2025.

\bibitem{dataverse-table}
Jason Cantarella, Andrew Rechnitzer, Henrik Schumacher, and Clayton Shonkwiler.
\newblock {Low stick number polygons representing all knot types through 13
  crossings}.
\newblock \url{https://doi.org/10.7910/DVN/NFJIII}, 2025.

\bibitem{cantarellaComputingConformalBarycenter2022}
Jason Cantarella and Henrik Schumacher.
\newblock Computing the conformal barycenter.
\newblock {\em SIAM Journal on Applied Algebra and Geometry}, 6(3):503--530,
  September 2022.

\bibitem{REAPR}
Jason Cantarella, Henrik Schumacher, and Clayton Shonkwiler.
\newblock Hard unknots are often easy from a different perspective.
\newblock In preparation.

\bibitem{snappy}
Marc Culler, Nathan~M. Dunfield, Matthias Goerner, and Jeffrey~R. Weeks.
\newblock Snap{P}y, a computer program for studying the geometry and topology
  of $3$-manifolds.
\newblock \url{https://snappy.computop.org}, 2009--2025.

\bibitem{diaoLowerBoundsLengths2003}
Yuanan Diao.
\newblock The lower bounds of the lengths of thick knots.
\newblock {\em Journal of Knot Theory and Its Ramifications}, 12(1):1--16,
  2003.

\bibitem{stick-knot-gen}
Thomas~D. Eddy.
\newblock stick-knot-gen, efficiently generate and classify random stick knots
  in confinement.
\newblock \url{https://github.com/thomaseddy/stick-knot-gen}.

\bibitem{eddyImprovedStickNumber2019}
Thomas~D. Eddy.
\newblock Improved {{Stick Number Upper Bounds}}.
\newblock Master's thesis, Colorado State University, Fort Collins, CO, USA,
  2019.
\newblock Available at \url{https://hdl.handle.net/10217/195411}.

\bibitem{eddyNewStickNumber2022}
Thomas~D. Eddy and Clayton Shonkwiler.
\newblock New stick number bounds from random sampling of confined polygons.
\newblock {\em Experimental Mathematics}, 31(4):1373--1395, 2022.

\bibitem{grosbergFlorytypeTheoryKnotted1996}
Alexander~Yu. Grosberg, Alexander Feigel, and Yitzhak Rabin.
\newblock Flory-type theory of a knotted ring polymer.
\newblock {\em Physical Review E}, 54(6):6618--6622, December 1996.

\bibitem{huhStickNumbers2bridge2011}
Youngsik Huh, Sungjong No, and Seungsang Oh.
\newblock Stick numbers of 2-bridge knots and links.
\newblock {\em Proceedings of the American Mathematical Society},
  139(11):4143--4152, 2011.

\bibitem{huhUpperBoundStick2011}
Youngsik Huh and Seungsang Oh.
\newblock An upper bound on stick number of knots.
\newblock {\em Journal of Knot Theory and Its Ramifications}, 20(5):741--747,
  2011.

\bibitem{jansevanrensburgBFACFstyleAlgorithmsPolygons2011}
Esias~J. {Janse van Rensburg} and Andrew Rechnitzer.
\newblock {{BFACF-style}} algorithms for polygons in the body-centered and
  face-centered cubic lattices.
\newblock {\em Journal of Physics A: Mathematical and Theoretical},
  44(16):165001, April 2011.

\bibitem{jansevanrensburgBFACFAlgorithmKnotted1991}
Esias~J. {Janse van Rensburg} and Stuart~G. Whittington.
\newblock The {{BFACF}} algorithm and knotted polygons.
\newblock {\em Journal of Physics A: Mathematical and General},
  24(23):5553--5567, December 1991.

\bibitem{jinPolygonIndicesSuperbridge1997}
Gyo~Taek Jin.
\newblock Polygon indices and superbridge indices of torus knots and links.
\newblock {\em Journal of Knot Theory and Its Ramifications}, 6(2):281--289,
  1997.

\bibitem{johnsonStickRamseyNumbers2013}
Maribeth Johnson, Stacy~Nicole Mills, and Rolland Trapp.
\newblock Stick and {{Ramsey}} numbers of torus links.
\newblock {\em Journal of Knot Theory and Its Ramifications}, 22(7):1350027,
  2013.

\bibitem{MR899057}
Louis~H. Kauffman.
\newblock State models and the {J}ones polynomial.
\newblock {\em Topology}, 26(3):395--407, 1987.

\bibitem{kuiperNewKnotInvariant1987}
Nicolaas~H. Kuiper.
\newblock A new knot invariant.
\newblock {\em Mathematische Annalen}, 278(1--4):193--209, 1987.

\bibitem{lipshitz-blender}
Robert Lipshitz.
\newblock {Mathematical Illustration in Blender}.
\newblock \url{https://pages.uoregon.edu/lipshitz/Blender/}, 2025.

\bibitem{knotinfo}
Charles Livingston and Allison~H. Moore.
\newblock {KnotInfo: Table of Knot Invariants}.
\newblock \url{https://knotinfo.org}, 2025.

\bibitem{millettKnottingRegularPolygons1994}
Kenneth~C. Millett.
\newblock Knotting of regular polygons in {$3$}-space.
\newblock {\em Journal of Knot Theory and Its Ramifications}, 3(3):263--278,
  1994.

\bibitem{millettMonteCarloExplorations2000}
Kenneth~C. Millett.
\newblock Monte {{Carlo}} explorations of polygonal knot spaces.
\newblock In Cameron~McA. Gordon, Vaughan F.~R. Jones, Louis~H. Kauffman, Sofia
  Lambropoulou, and J{\'o}zef~H. Przytycki, editors, {\em Knots in {{Hellas}}
  '98: {{Proceedings}} of the {{International Conference}} on {{Knot Theory}}
  and {{Its Ramifications}}}, volume~24 of {\em Series on {{Knots}} and
  {{Everything}}}, pages 306--334. World Scientific Publishing, Singapore,
  2000.

\bibitem{millettPhysicalKnotTheory2012}
Kenneth~C. Millett.
\newblock Physical knot theory: {{An}} introduction to the study of the
  influence of knotting on the spatial characteristics of polymers.
\newblock In Louis~H. Kauffman, Sofia Lambropoulou, Slavik Jablan, and
  J{\'o}zef~H. Przytycki, editors, {\em Introductory {{Lectures}} on {{Knot
  Theory}}}, volume~46 of {\em Series on {{Knots}} and {{Everything}}}, pages
  346--378. World Scientific Publishing, Singapore, 2012.

\bibitem{millettEnergyRopelengthOther2003}
Kenneth~C. Millett and Eric~J. Rawdon.
\newblock Energy, ropelength, and other physical aspects of equilateral knots.
\newblock {\em Journal of Computational Physics}, 186(2):426--456, 2003.

\bibitem{Wirt_Hm}
Nathaniel Morrison.
\newblock {Wirt\_Hm}.
\newblock \url{https://github.com/ThisSentenceIsALie/Wirt_Hm}, 2022--2025.

\bibitem{musickMinimalBridgeProjections2012}
Chad Musick.
\newblock Minimal bridge projections for 11-crossing prime knots.
\newblock Preprint, \href{https://arxiv.org/abs/1909.00917}{\tt arXiv:1208.4233
  [math.GT]}, 2012.

\bibitem{negamiRamseyTheoremsKnots1991}
Seiya Negami.
\newblock Ramsey theorems for knots, links and spatial graphs.
\newblock {\em Transactions of the American Mathematical Society},
  324(2):527--541, 1991.

\bibitem{randellConformationSpacesMolecular1988}
Richard Randell.
\newblock Conformation spaces of molecular rings.
\newblock In R.~C. Lacher, editor, {\em {{MATH}}/{{CHEM}}/{{COMP}} 1987},
  volume~54 of {\em Studies in {{Physical}} and {{Theoretical Chemistry}}},
  pages 141--156. Elsevier, Amsterdam, 1988.

\bibitem{randellMolecularConformationSpace1988}
Richard Randell.
\newblock A molecular conformation space.
\newblock In R.~C. Lacher, editor, {\em {{MATH}}/{{CHEM}}/{{COMP}} 1987},
  volume~54 of {\em Studies in {{Physical}} and {{Theoretical Chemistry}}}.
  Elsevier, Amsterdam, 1988.

\bibitem{rawdonUpperBoundsEquilateral2002}
Eric~J. Rawdon and Robert~G. Scharein.
\newblock Upper bounds for equilateral stick numbers.
\newblock In Jorge~Alberto Calvo, Kenneth~C. Millett, and Eric~J. Rawdon,
  editors, {\em Physical {{Knots}}: {{Knotting}}, {{Linking}}, and {{Folding
  Geometric Objects}} in {$\Bbb R^3$}}, volume 304 of {\em Contemporary
  {{Mathematics}}}, pages 55--75. American Mathematical Society, Providence,
  RI, USA, 2002.

\bibitem{andrew-minimal-knots}
Andrew Rechnitzer.
\newblock Minimal knots on cubic lattices.
\newblock \url{https://personal.math.ubc.ca/~andrewr/knots/minimal_knots.html},
  2011.

\bibitem{schareinInteractiveTopologicalDrawing1998}
Robert~G. Scharein.
\newblock {\em Interactive {{Topological Drawing}}}.
\newblock PhD thesis, University of British Columbia, Vancouver, BC, Canada,
  1998.

\bibitem{KnotPlot}
Robert~G. Scharein.
\newblock {KnotPlot}.
\newblock \url{https://knotplot.com}, 1998--2025.

\bibitem{schareinBoundsMinimumStep2009}
Robert~G. Scharein, Kai Ishihara, Javier Arsuaga, Yuanan Diao, Koya Shimokawa,
  and Mariel Vazquez.
\newblock Bounds for the minimum step number of knots in the simple cubic
  lattice.
\newblock {\em Journal of Physics A: Mathematical and Theoretical},
  42(47):475006--25, November 2009.

\bibitem{knoodle}
Henrik Schumacher and Jason Cantarella.
\newblock Knoodle.
\newblock \url{https://github.com/HenrikSchumacher/Knoodle}, 2025.

\bibitem{shonkwilerNewComputationsSuperbridge2020}
Clayton Shonkwiler.
\newblock New computations of the superbridge index.
\newblock {\em Journal of Knot Theory and Its Ramifications}, 29(14):2050096,
  2020.

\bibitem{shonkwilerAllPrimeKnots2022}
Clayton Shonkwiler.
\newblock All prime knots through 10 crossings have superbridge index {$\leq
  5$}.
\newblock {\em Journal of Knot Theory and Its Ramifications}, 31(4):2250023,
  2022.

\bibitem{shonkwilerNewSuperbridgeIndex2022a}
Clayton Shonkwiler.
\newblock New superbridge index calculations from non-minimal realizations.
\newblock {\em Journal of Knot Theory and Its Ramifications}, 31(10):2250063,
  September 2022.

\bibitem{pyknotid}
Alexander~J. Taylor and other SPOCK~contributors.
\newblock pyknotid knot identification toolkit.
\newblock \url{https://github.com/SPOCKnots/pyknotid}, 2014--2023.

\bibitem{thistlethwaiteEnumerationClassificationPrime2025}
Morwen~B. Thistlethwaite.
\newblock The enumeration and classification of prime 20-crossing knots.
\newblock {\em Algebraic \& Geometric Topology}, 25(1):329--344, March 2025.

\bibitem{ubertiMinimalLinksCubic1998}
Riccardo Uberti, Esias~J. {Janse van Rensburg}, Enzo Orlandini, Maria~Carla
  Tesi, and Stuart~G. Whittington.
\newblock Minimal links in the cubic lattice.
\newblock In Stuart~G. Whittington, De~Witt Sumners, and Timothy Lodge,
  editors, {\em The {{IMA Volumes}} in {{Mathematics}} and Its
  {{Applications}}}, pages 89--100. Springer New York, New York, NY, 1998.

\end{thebibliography}

\end{document}